\documentclass[a4paper,12pt, twoside, pdftex]{amsart}

\usepackage{fullpage}
\usepackage{amsmath, amsthm, amssymb}
\usepackage{txfonts}
\usepackage{amsfonts}
\usepackage{prettyref}
\usepackage{mathrsfs}
\usepackage[all]{xy}
\usepackage[utf8]{inputenc}

\newtheorem{dummy}{dummy}[section]
\newtheorem{lemma}[dummy]{Lemma}
\newtheorem{theorem}[dummy]{Theorem}

\newtheorem{conjecture}[dummy]{Conjecture}
\newtheorem{corollary}[dummy]{Corollary}

\theoremstyle{definition}
\newtheorem{definition}[dummy]{Definition}
\newtheorem{example}[dummy]{Example}
\newtheorem{remark}[dummy]{Remark}

\newtheorem{caution}[dummy]{Caution}

\newtheorem*{acknowledgements}{Acknowledgements}

\numberwithin{equation}{section}

\usepackage{tikz}
\usetikzlibrary{backgrounds}

\newrefformat{th}{Theorem~\ref{#1}}
\newrefformat{cr}{Corollary~\ref{#1}}
\newrefformat{lm}{Lemma~\ref{#1}}
\newrefformat{dl}{Definition-Lemma~\ref{#1}}
\newrefformat{df}{Definition~\ref{#1}}
\newrefformat{cl}{Claim~\ref{#1}}
\newrefformat{sl}{Sublemma~\ref{#1}}
\newrefformat{pr}{Proposition~\ref{#1}}
\newrefformat{cj}{Conjecture~\ref{#1}}
\newrefformat{st}{Step~\ref{#1}}
\newrefformat{sc}{Section~\ref{#1}}
\newrefformat{df}{Definition~\ref{#1}}
\newrefformat{rm}{Remark~\ref{#1}}
\newrefformat{q}{Question~\ref{#1}}
\newrefformat{pb}{Problem~\ref{#1}}
\newrefformat{cd}{Condition~\ref{#1}}
\newrefformat{example}{Example~\ref{#1}}
\newrefformat{he}{Heore~\ref{#1}}
\newrefformat{fg}{Figure~\ref{#1}}
\newrefformat{tb}{Table~\ref{#1}}
\newrefformat{as}{Assumption~\ref{#1}}

\newcommand{\pref}{\prettyref}

\newcommand{\DF}{\operatorname{DF}}

\newcommand{\Fun}{\operatorname{Fun}}

\newcommand{\id}{\operatorname{id}}

\newcommand{\Pos}{\operatorname{Pos}}
\newcommand{\Pro}{\operatorname{Pro}}
\newcommand{\Tw}{\operatorname{Tw}}

\newcommand{\cA}{\mathcal{A}}
\newcommand{\cB}{\mathcal{B}}
\newcommand{\cC}{\mathcal{C}}
\newcommand{\cD}{\mathcal{D}}

\newcommand{\cF}{\mathcal{F}}

\newcommand{\cL}{\mathcal{L}}

\newcommand{\cO}{\mathcal{O}}

\newcommand{\cS}{\mathcal{S}}

\newcommand{\cW}{\mathcal{W}}

\newcommand{\bZ}{\mathbb{Z}}

\newcommand{\bfw}{\mathbf{w}}

\newcommand{\bfD}{\mathbf{D}}
\newcommand{\bfE}{\mathbf{E}}

\newcommand{\bfK}{\mathbf{K}}

\newcommand{\bfQ}{\mathbf{Q}}

\newcommand{\frakf}{\mathfrak{f}}

\newcommand{\frakS}{\mathfrak{S}}

\begin{document}

\title{On the abstract wrapped Floer setups}
\author[H.~Morimura]{Hayato Morimura}
\address{Kavli Institute for the Physics and Mathematics of the Universe (WPI),
University of Tokyo,
5-1-5 Kashiwanoha,
Kashiwa,
Chiba,
277-8583,
Japan.}
\email{hayato.morimura@ipmu.jp}

%\subjclass[2010]{53D37}

\pagestyle{plain}

\begin{abstract}
The wrapped Fukaya category of a Liouville sector is defined via an axiomatic construction from the associated abstract wrapped Floer setup. In this paper, we propose a modified axiomatic construction, removing the irrelevant choices and the factorization axiom from the abstract wrapped Floer setup. Based on our modification, we reformulate and then prove a conjecture which essentially claims that the wrapped Fukaya category is obtained as the $\infty$-categorical localization along continuation maps. This amounts to compare the two axiomatic constructions.
\end{abstract}

\maketitle

%%%%%%%%%%%%%%%%%%%%%%%%%%%%%%%%%%%%%%%
\section{Introduction}
In their seminal works
\cite{GPS1, GPS2},
Ganatra--Pardon--Shende proposed a covariantly functorial construction of the (partially) wrapped Fukaya category of a (stopped) Liouville sector.
Based on the covariant functoriality,
they developed various
interesting
and
useful
theories.
One example is sectorial descent,
which provides us with a powerful local-to-global method to compute the wrapped Fukaya category of a Weinstein sector.
It also has applications in proving homological mirror symmetry
\cite{GJ, GL, Mor}.
See
\cite{GS1, GS2}
for the corresponding gluing method via
\cite{GPS3}
in microlocal sheaf theory.
Another is Viterbo restriction,
with respect to which the wrapped Fukaya category of a Weinstein sector exhibits Zariski-type descent
\cite{PS1, PS2, PS3, MSZ},
playing a key role in the proof of homological mirror symmetry.
See
\cite{Kuw}
for the corresponding gluing method via
\cite{GPS3}
in microlocal sheaf theory.

Their construction of the wrapped Fukaya category is twofold.
First,
to a stopped Liouville sector
$(X, \frakf)$,
one associates a collection of data
$(\cS_{(X, \frakf)}, H \cF^{\text{env}}_{\cS_{(X, \frakf)}}, \tilde{C}_{(X, \frakf)}, \{ \tilde{R}^{(X, \frakf)}_L \}_L )$,
called an abstract wrapped Floer setup,
so that
the assignments
\begin{align*}
(X, \frakf) \mapsto (\cS_{(X, \frakf)}, H \cF^{\text{env}}_{\cS_{(X, \frakf)}}, \tilde{C}_{(X, \frakf)}, \{ \tilde{R}^{(X, \frakf)}_L \}_L )
\end{align*}
are covariantly functorial with respect to suitable morphisms.
Next,
to an abstract wrapped Floer setup
$(\cS, H \cF^{\text{env}}_\cS, \tilde{C}, \{ \tilde{R}_L \}_L )$,
one associates a cofibrant strictly unital $A_\infty$-category
$\cW_{P_{\text{univ}}(\cS), \tilde{C}}$
so that
the assignments
\begin{align*}
(\cS, H \cF^{\text{env}}_\cS, \tilde{C}, \{ \tilde{R}_L \}_L ) \mapsto \cW_{P_{\text{univ}}(\cS), \tilde{C}}
\end{align*}
are covariantly functorial with respect to suitable morphisms.
The former step is geometric
while the latter is axiomatic.
In
\cite[Rem. 2.12]{GPS2},
Ganatra--Pardon--Shende conjectured the existence of a weaker axiomatic framework.
More specifically,
they expected that
one can remove from the latter step
the choices of
envelope
$H \cF^{\text{env}}_\cS$
for
$H \cF^{\text{pre}}_{\cS}$
and
wrapping categories
$\tilde{R}_L$
for all objects
$L$
of the $A_\infty$-pre-category
$\cF^{\text{pre}}_{\cS}$
associated with
$\cS$.

One of the two goals of this paper is to give an affirmative answer to the above conjecture.
Careful tracing of the original construction naturally leads us to a minimal collection of data
$(w \cS, C)$,
which we call a
\emph{weak abstract wrapped Floer setup}.
After refining the axiomatic composability
\cite[Def. 2.10(ii)]{GPS2} to obtain a modification
$w \cS$
of 
$\cS$,
this amounts to couple
$w \cS$
with a certain subset
$C \subset H^0 \cF_{w \cS}$
of the morphisms in canonical envelope
$H \cF_{w \cS}$
for
$H \cF^{\text{pre}}_{w \cS}$.
Note that
by definition
we have
$H \cF_{w \cS} \subset H \cF^{\text{env}}_{\cS}$
as $\bZ$-linear graded categories.
The weak abstract wrapped Floer setup
$(w \cS, C)$
is free from not only the irrelevant choices but also the factorization axiom
\cite[Def. 2.11(iv)]{GPS2}.
%Below in
%\pref{thm:comparison},
%the independence of
%$\cW_{P_{\text{univ}}(\cS), \tilde{C}}$
%from the choices will be precisely stated.
%While the difference from the original definition is subtle,
%one would realize in the end that
%our weaker version
%$(w \cS, C)$
%makes it clearer
%how the wrapped Fukaya category is obtained as the localization along continuation maps.
As desired,
also for weak abstract wrapped Floer setups there is a covariantly functorial construction of a cofibrant strictly unital $A_\infty$-category
whose cohomology recovers the wrapped Donaldson--Fukaya-type category.

%Thm
\begin{theorem}\label{thm:main}
There exist assignments
\begin{align} \label{eq:assignF}
w \cS \mapsto \cF^{\text{univ}}_{w \cS}
\end{align}
to cofibrant strictly unital $A_\infty$-categories with the following properties.
\begin{itemize}
\item[(i)]
When coupled with the set
$C \subset H^0 \cF_{w \cS}$
of continuation maps,
the assignments
\pref{eq:assignF}
extend to
\begin{align} \label{eq:assignW}
(w \cS, C) \mapsto \cF^{\text{univ}}_{w \cS}[C^{-1}_{\text{univ}}]
\end{align}
for some set
$C_{\text{univ}} \subset H^0 \cF^{\text{univ}}_{w \cS}$
uniquely determined by
$C$
up to unique isomorphism.
\item[(ii)]
The assignments
\pref{eq:assignW}
define an $\infty$-functor
\begin{align} \label{eq:functor}
N(\operatorname{wAWFS}) \to \cC at^s_{A_\infty}, \
(w \cS, C) \mapsto \cW_{w \cS, C} = \cF^{\text{univ}}_{w \cS}[C^{-1}_{\text{univ}}].
\end{align}
Here,
$N(\operatorname{wAWFS})$
denotes the nerve of the category of weak abstract wrapped Floer setups
and
$\cC at^s_{A_\infty}$
denotes the $\infty$-category of strictly unital $A_\infty$-categories.
\item[(iii)]
There exist canonical equivalences
\begin{align*}
H \cW_{w \cS, C} \simeq H \cW_{\DF, w \cS, C}
\end{align*}
to the weak wrapped Donaldson--Fukaya categories in the sense of Section
$4.3$,
which are compatible with the $\infty$-functor
\pref{eq:functor}.
\end{itemize}
\end{theorem}

%Rmk
\begin{remark}
As we will explain in Section
$4.3$,
essentially we do not add anything to the original definition.
Hence,
even after our modification,
usual geometric data still give sufficient input to define the wrapped Fukaya categories.
\end{remark}

The other goal is to give an affirmative answer to the following conjecture.

%Conj
\begin{conjecture}[{\cite[Conj. 2.23]{GPS2}}] \label{conj:2.23}
For
any $A_\infty$-category
$\cC$
and
suitably defined
$\Fun(\cF^{\text{pre}}_\cS, \cC)$,
the restriction functor
\begin{align*}
\Fun(\cW_{P_{\text{univ}}(\cS), \tilde{C}}, \cC) \to \Fun(\cF^{\text{pre}}_\cS, \cC) 
\end{align*}
is fully faithful
and
its essential image consists of those functors
which send continuation maps to isomorphisms in
$\cC$.
\end{conjecture}

Suppose that
$\cF^{\text{pre}}_\cS$
was a cofibrant unital $A_\infty$-category
and
$\cW_{P_{\text{univ}}(\cS), \tilde{C}}$
was its $\infty$-categorical localization
$\cF^{\text{pre}}_\cS[\tilde{C}^{-1}_\circ]$
along the set
$\tilde{C}_\circ \subset H^0 \cF^{\text{pre}}_\cS$
of continuation maps.
Then,
due to
\cite{OT25},
for any unital $A_\infty$-category
$\cD$
the pullback would give a unital $A_\infty$-functor
\begin{align*}
\Fun^u_{A_\infty}(\cW_{P_{\text{univ}}(\cS), \tilde{C}}, \cD) \to \Fun^u_{A_\infty}(\cF^{\text{pre}}_\cS, \cD) 
\end{align*}
between unital $A_\infty$-categories of unital $A_\infty$-functors with universal property:
it is fully faithful
and
its essential image consists of those functors
which send all representatives of any element in
$\tilde{C}_\circ$
to isomorphisms in
$\cD$.
As
$\cW_{P_{\text{univ}}(\cS), \tilde{C}}$
is the $\infty$-categorical localization
$\cO_{P_{\text{univ}}(\cS)}[I^{-1}_{P_{\text{univ}}(\cS), \tilde{C}}]$,
Conjecture
\pref{conj:2.23}
would be trivial
if we set
\begin{align*}
\Fun(\cF^{\text{pre}}_\cS, \cC) \coloneqq \Fun^u_{A_\infty}(\cO_{P_{\text{univ}}(\cS)}, \cC)
\end{align*}
and
assume
$\cC$
to be unital.
Here,
$I_{P_{\text{univ}}(\cS), \tilde{C}} \subset H^0 \cO_{P_{\text{univ}}(\cS)}$
is obtained by
somehow universally duplicating
$\tilde{C}$.
In other words,
$\cW_{P_{\text{univ}}(\cS), \tilde{C}}$
would be the $\infty$-categorical localization along universal duplicates of continuation maps.

Now,
one can consider the same thing as in the above paragraph replacing
$\cF^{\text{pre}}_\cS, \cW_{P_{\text{univ}}(\cS), \tilde{C}}$
with
$\cF^{\text{univ}}_{w \cS}, \cW_{w \cS, C}$,
where here
$\cW_{w \cS, C}$
is the $\infty$-categorical localization of the cofibrant (strictly) unital $A_\infty$-category
$\cF^{\text{univ}}_{w \cS}$
along
$C_{\text{univ}}$.
Hence it remains to compare the two axiomatic constructions.
If we
set
$\Fun(\cF^{\text{pre}}_\cS, \cC) \coloneqq \Fun^u_{A_\infty}(\cF^{\text{univ}}_{w \cS}, \cC)$
and
assume
$\cC$
to be unital,
then Conjecture
\pref{conj:2.23}
could be reformulated as

%Thm
\begin{theorem} \label{thm:comparison}
Let
$(w \cS, C)$
be
a weak abstract wrapped Floer setup
and 
$(\cS, H \cF^{\text{env}}_\cS, \tilde{C}, \{ \tilde{R}_L \}_L)$
an abstract wrapped Floer setup with
$\tilde{C} |_{H^0 \cF_{w \cS}} = C$.
Then there exists a quasi-equivalence
\begin{align*}
\bar{\tau} \colon \cW_{P_{\text{univ}}(\cS), \tilde{C}} \xrightarrow{qeq} \cW_{w \cS, C}
\end{align*}
such that
for any unital $A_\infty$-category
$\cD$
the pullback induces a zigzag of unital $A_\infty$-functors
\begin{align*}
\Fun^u_{A_\infty}(\cW_{P_{\text{univ}}(\cS), \tilde{C}}, \cD)
\xleftarrow{qeq}
\Fun^u_{A_\infty}(\cW_{w \cS, C}, \cD)
\hookrightarrow
\Fun^u_{A_\infty}(\cF^{\text{univ}}_{w \cS}, \cD)
\end{align*}
with universal property:
it is fully faithful
and
its essential image consists of those functors
which send all representatives of any element in
$C$
to isomorphisms in
$\cD$.
\end{theorem}

%Rmk
\begin{remark}
The existence of some noncanonical fully faithful functor
$\Fun^u_{A_\infty}(\cW_{P_{\text{univ}}(\cS), \tilde{C}}, \cD)
\hookrightarrow
\Fun^u_{A_\infty}(\cF^{\text{univ}}_{w \cS}, \cD)$
easily follows from
\cite[Prop. 3.47]{Tan}
once we find an equivalence
$H \cW_{\DF, \cS, \tilde{C}} \simeq H \cW_{\DF, w \cS, C}$.
A little more nontrivial part is to guarantee the property of its essential image.
\end{remark}

Passing to the $\infty$-category
$\cC at^u_{A_\infty}$
of unital $A_\infty$-categories,
which by
\cite[Thm. A]{COS}
is canonically equivalent to $\cC at^s_{A_\infty}$,
we obtain the following refinement.

%Cor
\begin{corollary} \label{cor:comparison}
In the same setting as in
\pref{thm:comparison},
there exists a canonical equivalence
\begin{align*}
\bar{\tau} \colon \cW_{P_{\text{univ}}(\cS), \tilde{C}} \xrightarrow{\sim} \cW_{w \cS, C}
\end{align*}
in
$\cC at^u_{A_\infty}$
such that
for any
$\cD \in \cC at^u_{A_\infty}$
the pullback induces a canonical $\infty$-functor
\begin{align*}
\hom_{\cC at^u_{A_\infty}}(\cW_{P_{\text{univ}}(\cS), \tilde{C}}, \cD)
\xrightarrow{\sim}
\hom_{\cC at^u_{A_\infty}}(\cW_{w \cS, C}, \cD)
\hookrightarrow
\hom_{\cC at^u_{A_\infty}}(\cF^{\text{univ}}_{w \cS}, \cD)
\end{align*}
with universal property:
it is fully faithful
and
its essential image consists of those functors
which send all representatives of any element in
$C$
to equivalence in
$\cD$.
\end{corollary}

Finally,
the sense becomes clear in
which the wrapped Fukaya category is obtained as the localization along continuation maps.
It might be an interesting problem to
construct the $B$-side counterparts of
$\cF^{\text{univ}}_{w \cS}$
and
continuation maps to lift homological mirror symmetry. 

\begin{acknowledgements}
The author was supported by
JSPS KAKENHI Grant Number
JP23KJ0341.
He would like to thank John Pardon for explaining about
the universal decorated poset
and
the proof of
\cite[Prop. 2.21]{GPS2}.
He is grateful to
Denis Auroux for a helpful suggestion
and
an anonymous referee for careful reading.
\end{acknowledgements}

%%%%%%%%%%%%%%%%%%%%%%%%%%%%%%%%%%%%%%%
\section{$\infty$-categorical quotients and localizations of $A_\infty$-categories}
In this section,
we review the construction of
$\infty$-categorical
quotients
and
localizations
of $A_\infty$-categories.
We denote by
$A_\infty Cat$
the ordinary category of small $\bZ$-linear
$A_\infty$-categories
and
$A_\infty$-functors.
For the rest of the paper,
we will assume that
all
$A_\infty$-categories
and
$A_\infty$-functors
belong to
$A_\infty Cat$.

\subsection{Unitalities}
One advantage to work within the $\infty$-category of $A_\infty$-categories is the equivalence of various unitalities.
Recall that
an $A_\infty$-category
$\cA$
is
\emph{strictly unital}
if it admits a closed degree
$0$
element
$\id_X \in \cA(X, X)$
such that
the chain maps
\begin{align*}
\mu^2_\cA(\id_X, -) \colon \cA(Y, X) \to \cA(Y, X), \
\mu^2_\cA(-, \id_X) \colon \cA(X, Y) \to \cA(X, Y)
\end{align*}
are the identities
for all
$X, Y \in \cA$
and
\begin{align*}
\mu^k_\cA(\id_X \otimes \cdots)
=
\cdots
=
\mu^k_\cA(\cdots \otimes \id_X \otimes \cdots)
=
\cdots
=
\mu^k_\cA(\cdots \otimes \id_X)
=
0
\end{align*}
whenever
$k > 1$
for all
$X \in \cA$.
We call such an element
\emph{strict unit}.
An $A_\infty$-functor
$F \colon \cA \to \cB$
is
\emph{strictly unital}
if
$\cA, \cB$
are strictly unital
and
we have
\begin{align*}
F^1(\id_X) = \id_{F^0(X)}, \
F^k(\id_X \otimes \cdots)
=
\cdots
=
F^k(\cdots \otimes \id_X \otimes \cdots)
=
\cdots
=
F^k(\cdots \otimes \id_X)
=
0
\end{align*}
for all
$X \in \cA$
and
$k > 1$.

Recall that
an $A_\infty$-category
$\cA$
is
\emph{cohomologically unital}
if its cohomology category
$H \cA$
is unital.
We call a closed degree
$0$
element
$e_X \in \cA(X, X)$
representing the unit in
$H \cA(X, X)$
a
\emph{cohomological unit}
of
$X$.
An $A_\infty$-functor
$F \colon \cA \to \cB$
is
\emph{cohomologically unital}
if
$\cA, \cB$
are cohomologically unital
and
$HF$
sends units to units.
Cohomologically unital
$A_\infty$-categories
and
$A_\infty$-functors
have the following refinements.

%Def
\begin{definition}[{\cite[Def. 7.3, 8.1, Prop. 7.4, 8.2]{Lyu03}}]
An $A_\infty$-category
$\cA$
is
\emph{unital}
if
it is cohomologically unital
and
the chain maps
\begin{align*}
\mu^2_\cA(e_X, -) \colon \cA(Y, X) \to \cA(Y, X), \
\mu^2_\cA(-, e_X) \colon \cA(X, Y) \to \cA(X, Y)
\end{align*}
are homotopic to the identities for all
$X, Y \in \cA$.
An $A_\infty$-functor
$F \colon \cA \to \cB$
is
\emph{unital}
if
$\cA, \cB$
are unital
and
$HF$
sends units to units.
\end{definition}

In fact,
over a filed,
an $A_\infty$-category
$\cA$
is unital
if and only if
it is cohomologically unital
\cite[Rem. 1.11]{COS}.
We denote by
$A_\infty Cat^s, A_\infty Cat^c$
and
$A_\infty Cat^u$
the subcategories of
$A_\infty Cat$
formed by
$A_\infty$-categories
and
$A_\infty$-functors
which are
strictly unital,
cohomologically unital
and
unital
respectively.
There are canonical faithful inclusions
\begin{align} \label{eq:inclusions}
A_\infty Cat^s
\to
A_\infty Cat^u
\hookrightarrow
A_\infty Cat^c
\end{align}
where the last one is also full. 
Let
$\bfw_s, \bfw_u$
and
$\bfw_c$
be the subcategories of
$A_\infty Cat^s, A_\infty Cat^u$
and
$A_\infty Cat^c$
generated by quasi-equivalences.
For
$* = s, u, c$
we denote by
$\cC at^*_{A_\infty}$
the $\infty$-categorical localizations
$N(A_\infty Cat^*)[N(\bfw_*)^{-1}]$.
They are the homotopy pushouts of the diagrams
$|N(\bfw_*)| \leftarrow N(\bfw_*) \to N(A_\infty Cat^*)$
where
$N(-)$
and
$|-|$
mean taking the
nerve
and
Kan completion
respectively.
Due to the following result,
we may call any of
$\cC at^*_{A_\infty}$
the
\emph{$\infty$-category of $A_\infty$-categories}.

%Thm
\begin{theorem}[{\cite[Thm. A]{COS}}]
The canonical functors
\pref{eq:inclusions}
induce equivalences of $\infty$-categories
\begin{align} \label{eq:A}
\cC at^s_{A_\infty}
\xrightarrow{\sim}
\cC at^u_{A_\infty}
\xrightarrow{\sim}
\cC at^c_{A_\infty}.
\end{align}
\end{theorem}

%Rmk
\begin{remark}
The first equivalence in
\pref{eq:A}
was also proved in
\cite{Tan}. 
Let
$dgCat$
be the ordinary category of
dg categories 
and
dg functors
and
$\cC at_{dg}$
the $\infty$-category of dg categories defined in the same way as
$\cC at^*_{A_\infty}$.
Then the canonical fully faithful inclusion
$dgCat \hookrightarrow A_\infty Cat$
induces an equivalence
$\cC at_{dg} \to \cC at_{A_\infty}$
of $\infty$-categories
\cite{COS, Pas, Tan}.
\end{remark}

\subsection{Cofibrancy}
Cofibrant $A_\infty$-categories behave well in
$\cC at^*_{A_\infty}$.
An $A_\infty$-category
$\cA$
is
\emph{cofibrant}
if for all
$X, Y \in \cA$
the morphism complexes
$\cA(X, Y)$
are $K$-projective
\cite[Def. 1.2]{OT25}.
Note that
there is no model structure on
$A_\infty Cat$.
Due to
\cite[Prop. 5.8]{Spa},
every $K$-projective cochain complex is $K$-flat.
By
\cite[Lem. 3.4]{GPS1}
it is cofibrant in the sense of
\cite[Def. 3.2]{GPS1}.
On the other hand,
every cofibrant cochain complex in the projective model structure is $K$-projective
\cite[Lem. 2.3.8]{Hov}.
As explained in
\cite[Rem. 1.6]{OT25},
one can always replace any morphism
$F \colon \cA \to \cB$
in
$\cC at^u_{A_\infty}$
with that
$\tilde{F} \colon \tilde{\cA} \to \tilde{\cB}$
of cofibrant unital $A_\infty$-categories giving a commutative diagram
$\Delta^1 \times \Delta^1 \to \cC at^u_{A_\infty}$
of the form
\begin{align*}
\begin{gathered}
\xymatrix{
\tilde{\cA} \ar^{\tilde{F}}[r] \ar[d] & \tilde{\cB} \ar[d] \\
\cA \ar^{F}[r] & \cB,
}
\end{gathered}
\end{align*}
where the vertical arrows are equivalences in
$\cC at^u_{A_\infty}$
induced by some quasi-equivalences.
We call such a morphism as
$\tilde{F}$
a
\emph{cofibrant replacement}
of
$F$.
By
\cite[Prop. 3.8]{Tan}
every quasi-isomorphism of cofibrant cochain complexes admits an inverse up to homotopy.
Moreover,
every quasi-equivalence of cofibrant $A_\infty$-categories admits an inverse up to natural equivalence
\cite[Prop. 3.45]{Tan}.
Restricting further to cofibrant unital $A_\infty$-categories,
one has the following useful results.

%Lem
\begin{lemma}[{\cite[Prop. 3.46]{Tan}}] \label{lem:3.46}
For any unital quasi-equivalence
$\cA \to \cA^\prime$
of cofibrant $A_\infty$-categories, 
the pullback
$\Fun^u_{A_\infty}(\cA^\prime, \cB)
\to
\Fun^u_{A_\infty}(\cA, \cB)$
between unital $A_\infty$-categories of unital $A_\infty$-functors induces an equivalence in
$\cC at^u_{A_\infty}$.
\end{lemma}

%Lem
\begin{lemma}[{\cite[Prop. 3.47]{Tan}}] \label{lem:lift}
Suppose that
there is an equivalence
$H \cA \simeq H \cA^\prime$
of cohomology categories for cofibrant unital $A_\infty$-categories
$\cA, \cA^\prime$.
Then it lifts to a unital $A_\infty$-functor
$\cA \to \cA^\prime$
inducing an equivalence in
$\cC at^u_{A_\infty}$.
\end{lemma}

%Rmk
\begin{remark}[{\cite[Rmk. 1.5]{OT25}, \cite[Rmk 1.6]{Tan}}]
Functor $A_\infty$-categories are not necessarily invariant under quasi-equivalences in the domain variable.
Specifically,
if either of
$\cA$
or
$\cA^\prime$
is not cofibrant,
then in general the pullback does not induce a bijection on homotopy classes of $A_\infty$-functors. 
\end{remark}

\subsection{Quotients and localizations}
Another advantage to work within
$\cC at^*_{A_\infty}$
is the rigorous universal property of
quotients
and
localizations of $A_\infty$-categories.
Let
$\cB \hookrightarrow \cA$
be a fully faithful inclusion of unital $A_\infty$-categories.
We denote by
$\mathbf{0}$
the zero category,
i.e.
the category with
a single object
and
the zero morphism.
Then the
\emph{quotient $A_\infty$-category}
$\cA / \cB$
is defined in
\cite[Def. 4.1]{OT25}
to be the pushout in
$\cC at^u_{A_\infty}$
of the diagram
$\mathbf{0} \leftarrow \cB \hookrightarrow \cA$.
Passing via the canonical equivalence to
$\cC at_{dg}$,
one sees that
$\cC at^u_{A_\infty}$
is presentable.
In particular,
$\cC at^u_{A_\infty}$
admits all colimits.
Hence
$\cA / \cB$
always exists.
Moreover,
by
\cite[Prop. 4.5]{OT25}
it exhibits universal property
when
$\cA$
is in addition cofibrant.
Namely,
for any unital $A_\infty$-category
$\cD$
the pullback
\begin{align*}
\hom_{\cC at^u_{A_\infty}}(\cA / \cB, \cD) \to \hom_{\cC at^u_{A_\infty}}(\cA, \cD)
\end{align*}
is fully faithful
and
its essential image consists of those functors
which are contractible along
$\cB$
\cite[Def. 4.2, Rmk. 4.3]{OT25}.
Due to the following result,
for any unital $A_\infty$-functor
$i \colon \cB \to \cA$
one may define the
\emph{quotient $A_\infty$-category}
$\cA / \cB$
along
$i$
as
$\cA / \cB^\#$,
where
$\cB^\# \subset \cA$
denotes the full subcategory spanned by the essential image of
$i$.

%Lem
\begin{lemma}[{\cite[Prop. 4.7]{OT25}}]
For any unital $A_\infty$-functor
$i \colon \cB \to \cA$,
the pushout in
$\cC at_{A_\infty}$
of the diagram
$\bf0 \leftarrow \cB \xrightarrow{i} \cA$
is equivalent to
$\cA / \cB^\#$.
\end{lemma}

One can construct localizations of $A_\infty$-categories via quotients.
Let
$\cA$
be a unital $A_\infty$-category
and
fix a subset
$W \subset H^0 \cA$
of cohomology classes of some morphisms in
$\cA$.
In
\cite[Def. 5.2]{OT25}
a
\emph{localization}
of
$\cA$
along
$W$
is defined to be a unital $A_\infty$-functor
$\iota \colon \cA \to \cL$
if any of its cofibrant replacements
$\tilde{\iota} \colon \tilde{\cA} \to \tilde{\cL}$
satisfies the following universal property in the sense of
\cite[Prop. 5.1(a)]{OT25}:
for every unital $A_\infty$-category
$\cD$,
the pullback
\begin{align*}
\tilde{\iota}^*
\colon
\hom_{\cC at^u_{A_\infty}}(\tilde{\cL}, \cD)
\to
\hom_{\cC at^u_{A_\infty}}(\tilde{\cA}, \cD)
\end{align*}
is fully faithful
and
its essential image consists of those functors
which send all representatives of any element in
$W$
to equivalences in
$\cD$.
We denote by
$\cA[W^{-1}]$
the codomain of such a morphism as
$\iota$,
which is also called the
\emph{localization}
of
$\cA$
along
$W$.
Let
$\cB_W \subset \cA$
be the full subcategory spanned by objects
which arise as mapping cones of representatives of elements in
$W$.
The quotient $A_\infty$-category
$(\Tw \cA) / \cB_W$
receives the canonical functor
$\cA \hookrightarrow \Tw \cA \to (\Tw \cA) / \cB_W$.
We denote by
$\cL_W$
the full subcategory spanned by its essential image.
Due to the following result,
it is a localization of
$\cA$
along
$W$
when
$\cA$
is in addition cofibrant.

%Thm
\begin{theorem}[{\cite[Thm. 5.9]{OT25}}] \label{thm:5.9}
When
$\cA$
is cofibrant,
the unital $A_\infty$-functor
$\cA \to \cL_W$
satisfies the universal property in the sense of
\cite[Prop. 5.1(a)]{OT25}.
\end{theorem}

\subsection{Models}
There are various models for the quotient $A_\infty$-categories.
Here,
we pick two
which will be relevant for us.
Let
$\cB \hookrightarrow \cA$
be a fully faithful inclusion of cofibrant unital $A_\infty$-categories.
We denote by
$\bfQ(\cA | \cB)$
the Lyubashenko--Manzyuk model from
\cite{LM}.
It exhibits universal property in the classical sense.

%Thm
\begin{theorem}[{\cite[Thm. 1.3]{LM}}]
There is a unital $A_\infty$-functor
$LM \colon \cA \to \bfQ(\cA | \cB)$
satisfying the following property:
for any unital $A_\infty$-category
$\cD$,
the pullback
\begin{align*}
LM^*
\colon
\Fun^u_{A_\infty}(\bfQ(\cA | \cB), \cD)
\to
\Fun^u_{A_\infty}(\cA, \cD)
\end{align*}
is fully faithful
and
its essential image consists of those functors
which are contractible along
$\cB$
\cite[Def. 4.2, Rmk. 4.3]{OT25}.
%Moreover,
%$LM^*$
%defines an equivalence in
%$\cC at^u_{A_\infty}$
%to its essential image
%$\Fun^u_{A_\infty}(\bfQ(\cA | \cB), \cD)^\#$.
\end{theorem}

We denote by
$\bfD(\cA | \cB)$
the Lyubashenko--Ovsienko model from
\cite{LO}.
As an object of
$\cC at^u_{A_\infty}$,
it coincides with
$\bfQ(\cA | \cB)$.

%Thm
\begin{theorem}[{\cite[Thm. 7.4]{LM}}]
There is a unital $A_\infty$-functor
$LO \colon \cA \to \bfD(\cA | \cB)$
satisfying the following property:
it is contractible along
$\cB$
\cite[Def. 4.2, Rmk. 4.3]{OT25}
and
the canonical unital $A_\infty$-functor from
$\bfQ(\cA | \cB)$
induces an equivalence in
$\cC at^u_{A_\infty}$.
\end{theorem}

Due to the following result,
one may adopt both
$\bfQ(\cA | \cB)$
and
$\bfD(\cA | \cB)$
as a model for
$\cA / \cB$,
when
$\cA$
is in addition cofibrant.

%Thm
\begin{theorem}[{\cite[Thm. 4.13]{OT25}}] \label{thm:4.13}
When
$\cA$
is cofibrant,
$LM$
and
$LO$
respectively exhibit
$\bfQ(\cA | \cB)$
and
$\bfD(\cA | \cB)$
as the quotient $A_\infty$-category
$\cA / \cB$.
\end{theorem}

Let
$\cA$
be a unital $A_\infty$-category
and
fix a subset
$W \subset H^0 \cA$
of cohomology classes of some morphisms in
$\cA$.
Combining
\pref{thm:5.9}
and
\pref{thm:4.13},
one obtains

%Cor
\begin{corollary}[{\cite[Thm. 5.13]{OT25}}] \label{cor:5.13}
The localization
$\cA[W^{-1}]$
is equivalent to the full subcategory of
$\bfD(\cA | \cB_W)$
spanned by objects in the essential image of the the canonical functor
$\cA \hookrightarrow \Tw \cA \to (\Tw \cA) / \cB_W \simeq \bfD(\cA | \cB_W)$.
\end{corollary}

As an application of Corollary
\pref{cor:5.13},
one obtains the following result,
which allows us to compute the morphism complexes in
$\cA[W^{-1}]$
as that in
$\bfD(\cA | \cB_W)$.
See also
\cite[Lem. 3.16]{GPS1}.

%Thm
\begin{theorem}[{\cite[Thm. 5.14]{OT25}}] \label{thm:5.14}
Let
$\cA$
be a cofibrant unital $A_\infty$-category.
Fix a sequence of objects
$\cdots \to X_1 \to X_0$
in
$\cA$
defining a functor
$\bZ_{\geq 0} \to N(\cA)$
from the partially ordered set of nonnegative integers to the $A_\infty$-nerve of
$\cA$.
Suppose that
for all representatives
$K \to K^\prime$
of any element of
$W$
the induced map
\begin{align*}
\varinjlim_i \cA(X_i, K) \to \varinjlim_i \cA(X_i, K^\prime)
\end{align*}
is a quasi-isomorphism.
Then for every
$Y \in \cA$
so is the induced map
\begin{align*}
\varinjlim_i \hom_\cA(X_i, Y) \to \varinjlim_i \hom_{\cA[W^{-1}]}(X_i, Y).
\end{align*}
\end{theorem}

%%%%%%%%%%%%%%%%%%%%%%%%%%%%%%%%%%%%%%%
\section{The original axiomatic construction}
In this section,
we review the axiomatic construction of the wrapped Fukaya categories developed by Ganatra--Pardon--Shende in
\cite[Sec. 2.3]{GPS2}.
For convenience of the reader,
we include the proof of
\cite[Lem. 2.18]{GPS2}.
What we newly add is
Caution
\pref{ctn:1}
and
Caution
\pref{ctn:2}.

\subsection{Abstract Floer setups}
%Def
\begin{definition}[{\cite[Def. 2.9]{GPS2}}] \label{dfn:cS}
An
\emph{abstract Floer setup}
$\cS$
consists of the following data.
\begin{itemize}
\item[(i)]
A set
$\cL$
of \emph{Lagrangians}.
\item[(ii)]
Subsets
$\cL_k \subset \cL^{k+1}$
of
\emph{composable tuples}
for integers
$k \geq 1$
which are closed under
passing to subsequences:
if
$(L_0, \ldots, L_k) \in \cL_k$
then
$(L_{i_0}, \ldots, L_{i_l}) \in \cL_l$
for every
$0 \leq i_0 < \cdots <i_l \leq k$.
\item[(iii)]
Graded modules
$CF(L, K)$
for all
$(L, K) \in \cL_1$.
\item[(iv)]
A contravariant functor
$D$
defined by specifying
sets
$D(L_0, \ldots, L_k)$
of
\emph{Floer data}
for all
$(L_0, \ldots, L_k) \in \cL_k$
and
restriction maps
$D(L_0, \ldots, L_k) \to D(L_{i_0}, \ldots, L_{i_l})$
along inclusions of subsequences for every
$0 \leq i_0 < \cdots <i_l \leq k$
with
$l \geq 1$.
\item[(v)]
Contractibility of Floer data:
the map from
$D(L_0, \ldots, L_k)$
to the limit of
$D$
applied to all proper subsequences of
$(L_0, \ldots, L_k)$
must be surjective.
In particular,
$D(L_0, L_1)$
is nonempty
as
$D(L)$
consists of a single point for every
$L \in \cL$.
\item[(vi)]
For every element
$\delta(L_0, \ldots, L_k) \in D(L_0, \ldots, L_k)$
a linear map
\begin{align*}
\mu^k_{\delta(L_0, \ldots, L_k)} \colon CF(L_0, L_1) \otimes \cdots \otimes CF(L_{k-1},L_k) \to CF(L_0, L_k)[2-k]
\end{align*}
satisfying the
$A_\infty$-relations with respect to the restriction maps on
$D$.
In particular,
$(CF(L_0, L_1), \mu^1_{\delta(L, K)})$
is a cochain complex for every
$(L, K) \in \cL_1$
and
$\delta(L, K) \in D(L, K)$.
We require each such complex to be cofibrant in the sense of
\cite[Def. 3.2]{GPS1}.
\item[(vii)]
Sets
$D^\prime(L, K), D^{\prime \prime }(L, K)$
and
$D^{\prime \prime \prime}_i(L_0, L_1, L_2)$
for
$i = 0, 1, 2$
together with surjective maps respectively to
$D(L, K) \times D(L, K), D^\prime(L, K) \times_{D(L, K) \times D(L, K)} (D^\prime(L, K) \times_{D(L, K)} D^\prime(L, K))$
and
\begin{align*}
\begin{gathered}
D(L_0, L_1, L_2) \times_{D(L_0, L_1) \times D(L_1, L_2) \times D(L_0, L_2)} (D(L_0, L_1, L_2) \times_{D(L_0, L_1)} D^\prime(L_0, L_1)), \\
D(L_0, L_1, L_2) \times_{D(L_0, L_1) \times D(L_1, L_2) \times D(L_0, L_2)} (D(L_0, L_1, L_2) \times_{D(L_1, L_2)} D^\prime(L_1, L_2)), \\
D(L_0, L_1, L_2) \times_{D(L_0, L_1) \times D(L_1, L_2) \times D(L_0, L_2)} (D(L_0, L_1, L_2) \times_{D(L_0, L_2)} D^\prime(L_0, L_2)).
\end{gathered}
\end{align*}
\item[(viii)]
For every
$\delta^\prime(L, K) \in D^\prime(L, K)$
a chain map
\begin{align*}
\alpha_{\delta^\prime(L, K)} \colon CF(L, K) \to CF(L, K)
\end{align*}
where the
domain
and
codomain
are equipped with the differentials associated with the image of
$\delta^\prime(L, K)$
under the above surjective map from
$D^\prime(L, K)$.
For every
$\delta^{\prime \prime}(L, K) \in D^{\prime \prime}(L, K)$
a chain homotopy
\begin{align*}
\beta_{\delta^{\prime \prime}(L, K)} \colon CF(L, K) \to CF(L, K)[-1]
\end{align*}
between an
$\alpha$
map
and
the composition of two
$\alpha$
maps associated with the image of
$\delta^{\prime \prime}(L, K)$
under the above surjective map from
$D^{\prime \prime}(L, K)$.
For every
$\delta^{\prime \prime \prime}(L_0, L_1, L_2) \in D^{\prime \prime \prime}_i(L_0, L_1, L_2)$
a chain homotopy
\begin{align*}
\gamma_{\delta^{\prime \prime \prime}(L_0, L_1, L_2)} \colon CF(L_0, L_1) \otimes CF(L_1, L_2) \to CF(L_0, L_2)[-1]
\end{align*}
between a
$\mu^2$
map
and
the composition of an
$\alpha$
with a
$\mu^2$
map associated with the image of
$\delta^{\prime \prime \prime}(L_0, L_1, L_2)$
under the above surjective map from
$D^{\prime \prime \prime}(L_0, L_1, L_2)$.
\item[(ix)]
A map
$f \colon D(L, K) \to D^{\prime}(L, K)$
whose composition with the above surjective map from
$D^\prime(L, K)$
is the diagonal
and
such that
every map
$\alpha_{f(\delta(L, K))}$
is the identity.
\end{itemize}
\end{definition}

We denote by
$\cF^{\text{pre}}_\cS$
the $A_\infty$-pre-category specified by data
(i)-(iv)
and
(vi).
As mentioned in
\cite[Rem. 2.7]{GPS2},
data
(i), (ii), (iv)
correspond to the semisimplicial set from
\cite[Def. 2.6]{GPS2}.
Note that
datum
(v)
allows one to construct a compatible collection of Floer data inductively.
The other data force all possibilities of such collection to define a single cohomology pre-category
$H \cF^{\text{pre}}_\cS$
up to unique equivalence.

%Def
\begin{definition}[{\cite[Def. 2.10]{GPS2}}] \label{dfn:2.10}
The
\emph{Donaldson--Fukaya pre-category}
$h \cF^{\text{pre}}_{\DF, \cS}$
of an abstract Floer setup
$\cS$
consists of the following data.
\begin{itemize}
\item[(i)]
The set
$\cL$
of Lagrangians from
$\cS$.
\item[(ii)]
The subsets
$\cL_k \subset \cL^{k+1}$
of
composable tuples from
$\cS$.
\item[(iii)]
The homotopy classes
$hF(L, K)$
of
$CF(L, K)$
from
$\cS$
for all
$(L, K) \in \cL_1$.
\item[(iv)]
For every
$(L_0, L_1, L_2) \in \cL_2$
a composition map
$hF(L_0, L_1) \otimes hF(L_1, L_2) \to hF(L_0, L_2)$.
\item[(v)]
For every
$(L_0, L_1, L_2, L_3) \in \cL_3$
the two maps from
$hF(L_0, L_1) \otimes hF(L_1, L_2) \otimes hF(L_2, L_3)$
to
$hF(L_0, L_3)$
obtained by composing in either order agree.
\end{itemize}
These are constructed from the data of
$\cS$
as follows.
For
$(L, K) \in \cL_1$
every
$\delta(L, K) \in D(L, K)$
determines a differential
$\mu^1_{\delta(L, K)}$
on
$CF(L, K)$.
If
$\delta_1(L, K), \delta_2(L, K) \in D(L, K)$
then there are
$\delta^\prime_{12}(L, K), \delta^\prime_{21}(L, K) \in D^\prime(L, K)$
which determine chain maps
\begin{align*}
\begin{gathered}
\alpha_{\delta^\prime_{12}(L, K)} \colon (CF(L, K), \mu^1_{\delta_1(L, K)})  \to (CF(L, K), \mu^1_{\delta_2(L, K)}), \\
\alpha_{\delta^\prime_{21}(L, K)} \colon (CF(L, K), \mu^1_{\delta_2(L, K)})  \to (CF(L, K), \mu^1_{\delta_1(L, K)}).
\end{gathered}
\end{align*}
By data
(vii), (viii)
in Definition
\pref{dfn:cS}
there are
$\delta^{\prime \prime}_{12}(L, K), \delta^{\prime \prime}_{21}(L, K) \in D^{\prime \prime}(L, K)$
such that
$\beta_{\delta^{\prime \prime}_{12}(L, K)}, \beta_{\delta^{\prime \prime}_{21}(L, K)}$
give chain homotopies between
$\alpha_{\delta^\prime_{21}(L, K)} \circ \alpha_{\delta^\prime_{12}(L, K)}, \alpha_{\delta^\prime_{12}(L, K)} \circ \alpha_{\delta^\prime_{21}(L, K)}$
and
the identities
$\alpha_{f_1(\delta_1(L, K))}, \alpha_{f_2(\delta_2(L, K))}$
for the maps
$f_1, f_2 \colon D(L, K) \to D(L, K)$
from datum
(ix).
This defines a cofibrant object
$hF(L, K)$
in the homotopy category of complexes up to unique equivalence.
For
$(L_0, L_1, L_2) \in \cL_2$
every
$\delta_{012}(L_0, L_1, L_2) \in D(L_0, L_1, L_2)$
determines a composition map
\begin{align*}
\mu^2_{\delta_{012}(L_0, L_1, L_2)} \colon CF(L_0, L_1) \otimes CF(L_1, L_2) \to CF(L_0, L_2).
\end{align*}
If
$\delta^\prime_{01}(L_0, L_1) \in D^{\prime}(L_0, L_1), \delta^\prime_{12}(L_1, L_2) \in D^\prime(L_1, L_2)$
and
$\delta^\prime_{02}(L_0, L_2) \in D^\prime(L_0, L_2)$
then there exist
$\delta^{\prime \prime \prime}_{01}(L_0, L_1, L_2) \in D^{\prime \prime \prime}_0(L_0, L_1, L_2), \delta^{\prime \prime \prime}_{12}(L_0, L_1, L_2) \in D^{\prime \prime \prime}_1(L_0, L_1, L_2)$
and
$\delta^{\prime \prime \prime}_{02}(L_0, L_1, L_2) \in D^{\prime \prime \prime}_2(L_0, L_1, L_2)$
which determine chain homotopies
$\gamma_{\delta^{\prime \prime \prime}_{01}(L_0, L_1, L_2)},
\gamma_{\delta^{\prime \prime \prime}_{12}(L_0, L_1, L_2)},
\gamma_{\delta^{\prime \prime \prime}_{02}(L_0, L_1, L_2)}$
between
$\mu^2_{\delta_{012}(L_0, L_1, L_2)}$
and
\begin{align*}
\mu^2_{\delta_{012,0}(L_0, L_1, L_2)} \circ (\alpha_{\delta^\prime_{01}(L_0, L_1)} \otimes \id), \
\mu^2_{\delta_{012,1}(L_0, L_1, L_2)} \circ (\id \otimes \alpha_{\delta^\prime_{12}(L_1, L_2)}), \
\alpha_{\delta^\prime_{02}(L_0, L_2)} \circ \mu^2_{\delta_{012,2}(L_0, L_1, L_2)}
\end{align*}
respectively for some
$\delta_{012,0}(L_0, L_1, L_2), \delta_{012,1}(L_0, L_1, L_2), \delta_{012,2}(L_0, L_1, L_2) \in D(L_0, L_1, L_2)$.
Hence a composition map descends to that of homotopy classes.
The associativity follows from datum
(vi)
in Definition
\pref{dfn:cS}.
\end{definition}

\subsection{Abstract wrapped Floer setups}
%Def
\begin{definition}[{\cite[Def. 2.11]{GPS2}}] \label{dfn:2.11}
An
\emph{abstract wrapped Floer setup}
$(\cS, H \cF^{\text{env}}_\cS, \tilde{C}, \{ \tilde{R}_L \}_L)$
is an abstract Floer setup
$\cS$
together with the following.
\begin{itemize}
\item[(i)]
An
\emph{envelope}
$H \cF^{\text{env}}_\cS$
for the cohomology pre-category
$H \cF^{\text{pre}}_\cS$:
a small category with object set
$\cL$
from
$\cS$
enriched over graded modules
$H \cF^{\text{env}}_\cS(L, K)$
for all
$L, K \in \cL$,
which coincide with
$H \cF^{\text{pre}}_\cS(L, K)$
if
$(L, K) \in \cL_1$.
\item[(ii)]
A set
$\tilde{C}$
of
\emph{continuation maps},
morphisms in
$H^0 \cF^{\text{env}}_\cS$
closed under composition.
\item[(iii)]
For all
$L \in \cL$
the
\emph{wrapping category}
$\tilde{R}_L$
of
$L$,
a filtered category of countable cofinality together with a functor
$\tilde{R}_L \to ((H^0 \cF^{\text{env}}_\cS)_{/_{\tilde{C}} L})^{op}$.
Here,
$(H^0 \cF^{\text{env}}_\cS)_{/_{\tilde{C}} L}$
is the
\emph{continuation slice category}
of
$L$,
i.e.
the full subcategory of the slice category
$(H^0 \cF^{\text{env}}_\cS)_{/ L}$
spanned by continuation maps
$L^w \to L$.
We define
\begin{align*}
HW_{\cS, \tilde{C}} (L, K) = \varinjlim_{L^w \in \tilde{R}_L} H \cF^{\text{env}}_\cS(L^w, K)
\end{align*}
and
require it to satisfy the
\emph{right locality property}:
the map
$HW_{\cS, \tilde{C}}(L, K) \to HW_{\cS, \tilde{C}}(L, K^\prime)$
given by multiplying on the right by any continuation map
$K \to K^\prime$
is an isomorphism.
\item[(iv)]
Every morphism
$A \leadsto B$
in
$\tilde{R}_L$
must satisfy the
\emph{factorization property}:
for all finite sets
$\bfK = \{ (K^j_1, \ldots, K^j_{a_j}) \}_j$
of composable tuples,
there exists a factorization
$A \leadsto H \leadsto B$
in
$\tilde{R}_L$
such that
each
$(H, K^j_1, \ldots, K^j_{a_j})$
with
$H$
regarded as the image in
$\cL$
belongs to
$\cL_{a_j}$.
Moreover,
we require that
there always are uncountably many possibilities for the Lagrangian in
$\cL$
corresponding to
$H$.
\end{itemize}
\end{definition}

%Def
\begin{definition}[{\cite[Def. 2.13]{GPS2}}] \label{dfn:2.13}
The
\emph{wrapped Donaldson--Fukaya category}
$H \cW_{\DF, \cS, \tilde{C}}$
of an abstract wrapped Floer setup
$(\cS, H \cF^{\text{env}}_\cS, \tilde{C}, \{ \tilde{R}_L \}_L)$
is the full subcategory of
$\Pro H \cF^{\text{env}}_\cS$
spanned by
$\tilde{R}_L$
for all
$L \in \cL$.
In other words,
the objects of
$H \cW_{\DF, \cS, \tilde{C}}$
are Lagrangians in
$\cL$
from
$\cS$
and
for every
$L, K \in \cL$
the morphisms are
\begin{align*}
H \cW_{\DF, \cS, \tilde{C}}(L, K)
=
\varprojlim_{K^w \in \tilde{R}_K} \varinjlim_{L^w \in \tilde{R}_L} H\cF^{\text{env}}_\cS(L^w, K^w)
=
HW_{\cS, \tilde{C}}(L, K)    
\end{align*}
with the evident notion of composition.
\end{definition}

There is a canonical functor
$H \cF^{\text{env}}_\cS \to H \cW_{\DF, \cS, \tilde{C}}$
induced by the maps
\begin{align*}
H \cF^{\text{env}}_\cS(L, K) \to H \cF^{\text{env}}_\cS(L^w, K) \to HW_{\cS, \tilde{C}}(L, K)
\end{align*}
for
all
$L, K \in \cL$
and
any
$L^w \in \tilde{R}_L$.
Since
$HW_{\cS, \tilde{C}}(L, K)$
satisfy the right locality property,
the maps are compatible with composition.

%Lem
\begin{lemma}[{\cite[Lem. 2.14]{GPS2}}] \label{lem:2.14}
For any category
$\cC$,
the pullback along the canonical functor
\begin{align*}
\Fun(H \cW_{\DF, \cS, \tilde{C}}, \cC) \to \Fun(H \cF^{\text{env}}_\cS, \cC)
\end{align*}
between ordinary functor categories
is fully faithful
and
its essential image consists of those functors
which send continuation maps to isomorphisms.
\end{lemma}

\subsection{Universal decorated posets}
%Def
\begin{definition}[{\cite[Def. 2.8]{GPS2}}]
A
\emph{decorated poset}
$\eta \colon P \to \cF^{\text{pre}}_\cS$
for an abstract Floer setup
$\cS$
is a poset
$P$
together with a map
$\eta$
from the nerve of
$P$
to the semisimplicial set associated with
$\cF^{\text{pre}}_\cS$.
The map
$\eta$
sends each totally ordered tuple
$p_k > \cdots > p_0$
in
$P$
to a pair of
a composable tuple
$(L_{p_k}, \ldots, L_{p_0}) \in \cL_k$
and
an element
$\eta(p_k, \cdots, p_0) \in D(L_{p_k}, \ldots, L_{p_0})$.
\end{definition}

In particular,
$\eta$
sends each element
$p \in P$
to a Lagrangian
$L_p \in \cL$.
We denote by
$\cO_P$
the cofibrant strictly unital $A_\infty$-category 
whose objects are elements
$p \in P$
and
for every
$p, q \in P$
whose morphisms are
\begin{align*}
\cO_P(p, q) =
\begin{cases}
CF(L_p, L_q)  & p > q, \\
\bZ \langle 1_p \rangle & p = q, \\
0 & \text{else}.
\end{cases}
\end{align*}
Here,
the $A_\infty$-operations for totally ordered tuples
$p_k > \cdots > p_0$
are
$\mu^k_{\eta(p_k, \ldots, p_0)}$
and
we require all
$1_p$
to be a strict unit.
Let
$I_{P, \tilde{C}}$
be the set of morphisms in
$H^0 \cO_P$
determined by
$P$
and
the set
$\tilde{C} \subset H^0 \cF^{\text{env}}_{\cS}$
of continuation maps.
We denote by
$\cW_{P, \tilde{C}}$
the localization
$\cO_P[I^{-1}_{P, \tilde{C}}]$.

%Def
\begin{definition}
A decorated poset
$P$
has
\emph{no duplicates}
if
$P^{\leq p}, P^{\leq q}$
are not isomorphic as decorated posets
whenever
$p \neq q$.
An
\emph{inclusion}
$P \subset P^\prime$
of decorated posets is an inclusion of underlying posets
such that
the decoration
$\eta$
on
$P$
is the restriction of that
$\eta^\prime$
on
$P^\prime$.
It is
\emph{downward closed}
if in addition
$P^{\prime \leq p} = P^{\leq p}$
for all
$p \in P$.
\end{definition}

Given any two decorated posets with no duplicates
$P, P^\prime$,
there is at most one downward closed inclusion
$P \subset P^\prime$.
Recall that
a poset
$P$
is
\emph{cofinite}
if
$P^{\leq p}$
is finite for all
$p \in P$.
Let
$\Pos_\cS$
be the ordinary category
whose objects are countable cofinite decorated posets for
$\cS$
with no duplicates
and
whose morphisms are downward closed inclusions.
In
\cite[Lem. 3.42]{GPS1}
Ganatra--Pardon--Shende introduced the universal decorated poset
\begin{align*}
\eta_{\text{univ}}(\cS)
\colon
P_{\text{univ}}(\cS)
\to
\cF^{\text{pre}}_\cS
\end{align*}
which corresponds to the full subcategory of
$\Pos_\cS$
spanned by objects of the form
$P^{\leq p}$.
By construction for every
$p \in P_{\text{univ}}(\cS)$
there exists some
$P^{\leq p} \in \Pos_\cS$
with
$P_{\text{univ}}(\cS)^{\leq p} = P^{\leq p}$.
According to
\cite[Sec. 2.3]{GPS2},
each
$p \in P_{\text{univ}}(\cS)$
corresponds to an isomorphism class
$[P^{\leq p}]$
of countable cofinite decorated posets with no duplicates.
Moreover,
every
$P \in \Pos_\cS$
admits a unique downward closed inclusion
$P \subset P_{\text{univ}}(\cS)$.

%Rmk
\begin{remark} \label{rmk:ss2}
Here,
the isomorphism class
$[P^{\leq p}]$
must mean that
any two representatives have isomorphic underlying posets each pair of
whose corresponding nerves map to the same image.
Otherwise,
the image of
$\eta_{\text{univ}}(\cS)$
would not cover
distinct Lagrangians
$L, K \in \cL$
nor a composable tuple
$(L_0, \cdots, L_k) \in \cL_k$
with distinct Floer data
$\delta(L_0, \ldots, L_k), \delta^\prime(L_0, \ldots, L_k) \in D(L_0, \ldots, L_k)$.
Note that
two distinctive decorated posets with a common underlying poset are not isomorphic in
$\Pos_\cS$,
as by definition of inclusions of decorated posets the restriction of the decoration on one must coincide with that on the other.
\end{remark}

%Ctn
\begin{caution} \label{ctn:1}
As mentioned above,
for any
$p \in P_{\text{univ}}(\cS)$
we have
$P^{\leq p}_{\text{univ}}(\cS) = P^{\leq p}$
for some
$P^{\leq p} \in \Pos_\cS$.
Let
$P^{\leq p}_\circ \subset P^{\leq p}$
be a decorated subposet obtained by removing one of the minimum elements,
which makes sense
as
$P_{\text{univ}}(\cS)$
is cofinite.
While the canonical inclusion
$P^{\leq p} \subset P^{\leq p}_{\text{univ}}(\cS)$
is downward closed,
the one
$P^{\leq p}_\circ \subset P^{\leq p}_{\text{univ}}(\cS)$
is not
since
$P^{\leq p}_{\text{univ}}(\cS)$
contains the removed element.
\end{caution}

%Def
\begin{definition}[{\cite[Def. 2.20]{GPS2}}]
The
\emph{wrapped Fukaya category}
$\cW_{\cS, \tilde{C}}$
of an abstract wrapped Floer setup
$(\cS, H \cF^{\text{env}}_\cS, \tilde{C}, \{ \tilde{R}_L \}_L)$
is
$\cW_{P_{\text{univ}}(\cS), \tilde{C}}$.
\end{definition}

\subsection{Wrapping sequences}
%Def
\begin{definition}[{\cite[Def. 2.15]{GPS2}}]
Let
$\eta \colon P \to \cF^{\text{pre}}_\cS$
be a decorated poset for
$\cS$.
For an abstract wrapped Floer setup
$(\cS, H \cF^{\text{env}}_\cS, \tilde{C}, \{ \tilde{R}_L \}_L)$
a
\emph{$P$-wrapping sequence}
is a cofinal sequence
$p_0 < p_1 < \cdots$
in
$P$
along with elements of
$I_{P, \tilde{C}}$
in
$H^0 \cO_P(p_{i+1}, p_i) = H^0 \cF^{\text{env}}_\cS(L_{p_{i+1}}, L_{p_i})$
such that
the canonical map
\begin{align*}
\varinjlim_i H \cF^{\text{env}}_\cS(L_{p_i}, K)
\to
\varinjlim_i HW_{\cS, \tilde{C}}(L_{p_i}, K)
=
HW_{\cS, \tilde{C}}(L_{p_0}, K)
\end{align*}
is an isomorphism for every
$K \in \cL$.
A decorated poset
$P$
for
$\cS$
is
\emph{sufficiently wrapped}
with respect to
$(\cS, H \cF^{\text{env}}_\cS, \tilde{C}, \{ \tilde{R}_L \}_L)$
if every
$p \in P$
belongs to a $P$-wrapping sequence for
$(\cS, H \cF^{\text{env}}_\cS, \tilde{C}, \{ \tilde{R}_L \}_L)$.
\end{definition}

Given any isomorphism
$L_p \xrightarrow{\sim} L_q$
in
$H \cW_{\DF, \cS, \tilde{C}}$
and
a
$P$-wrapping sequence
$p = p_0 < p_1 < \cdots$,
there is an element of
$I_{P, \tilde{C}}$
in
$H^0 \cO_P(p_i, q)$
for some
$p_i \in P$
which is sent to the isomorphism.
Hence
$L_p \cong L_q$
in
$H \cW_{\DF, \cS, \tilde{C}}$
implies that
$p$
belongs to a $P$-wrapping sequence for
$(\cS, H \cF^{\text{env}}_\cS, \tilde{C}, \{ \tilde{R}_L \}_L)$
if and only if
so does
$q$.

%Lem
\begin{lemma}[{\cite[Lem. 2.16]{GPS2}}] \label{lem:2.16}
Let
$p_0 < p_1 < \cdots$
be any $P$-wrapping sequence for
$(\cS, H \cF^{\text{env}}_\cS, \tilde{C}, \{ \tilde{R}_L \}_L)$.
Then the canonical map
\begin{align*}
\varinjlim_i H \cO_P(p_i, q) \to \varinjlim_i H \cW_{P, \tilde{C}}(p_i, q)
\end{align*}
is an isomorphism.
\end{lemma}

%Lem
\begin{lemma}[{\cite[Prop. 2.17]{GPS2}}] \label{lem:2.17}
For every inclusion
$P \subset P^\prime$
of sufficiently wrapped decorated posets
$P$
with respect to
$(\cS, H \cF^{\text{env}}_\cS, \tilde{C}, \{ \tilde{R}_L \}_L)$,
there are canonical fully faithful functors
\begin{align} \label{eq:can}
H \cW_{P, \tilde{C}} \hookrightarrow H \cW_{\DF, \cS, \tilde{C}} \hookleftarrow H \cW_{P^\prime, \tilde{C}}
\end{align}
compatible with the one
$H \cW_{P, \tilde{C}} \hookrightarrow H \cW_{P^\prime, \tilde{C}}$.
\end{lemma}

Due to the following result,
if there are countably many isomorphism classes of Lagrangians in
$H \cW_{\DF, \cS, \tilde{C}}$,
then
one can construct a sufficiently wrapped decorated poset
$\eta \colon P \to \cF^{\text{pre}}_\cS$
with respect to
$(\cS, H \cF^{\text{env}}_\cS, \tilde{C}, \{ \tilde{R}_L \}_L)$
for which the canonical functor
\pref{eq:can}
becomes an equivalence.

%Lem
\begin{lemma}[{\cite[Lem. 2.18]{GPS2}}] \label{lem:2.18}
Every countable cofinite decorated poset
$P$
for
$\cS$
admits a downward closed inclusion into a countable cofinite decorated poset
$P^\prime$
which is sufficiently wrapped with respect to
$(\cS, H \cF^{\text{env}}_\cS, \tilde{C}, \{ \tilde{R}_L \}_L)$.
\end{lemma}
\begin{proof}
Let
$P^\prime = P \sqcup \bZ_{\geq 0}$
be a countable cofinite poset defined as follows.
The ordering on
$P$
is the given one
and
that on
$\bZ_{\geq 0}$
is the usual one.
Choose a sequence of finite downward closed subposets
$S_0 \subset S_1 \subset \cdots P$
whose union is
$P$.
Add an ordering
such that
each
$i \in \bZ_{\geq 0}$
is larger than all elements in
$S_i$.
This additional ordering guarantees that
$P^{\prime \leq i}$
are finite for all
$i \in \bZ_{\geq 0}$.
For each
$p \in P$
choose a cofinal sequence
$L^0_p \leadsto L^1_p \leadsto \cdots$
in
$\tilde{R}_{L_p}$.
Choose a map
$\bZ_{\geq 0} \to P$
whose fiber
$i_{p, 0} < i_{p, 1} < \cdots$
over every
$p \in P$
is infinite.
Assign to
$i_{p, k}$
the Lagrangian
$H^k_p$
arising from a choice of factorization
$L^k_p \leadsto H^k_p \leadsto L^{k+1}_p$
in
$\tilde{R}_{L_p}$.
One chooses these factorizations by induction on
$i \in \bZ_{\geq 0}$,
using datum
(iv)
in Definition
\pref{dfn:2.11}
so that
all totally ordered tuples of
$P^\prime$
correspond to composable Lagrangians.
Now,
every
$p \in S_j \subset P \subset P^\prime$
belongs to a sequence
$p < i_{p, k} < i_{p, k+1} < \cdots$
in
$P^\prime$
for some
$k \in \bZ_{\geq 0}$
with
$j < i_{p, k}$,
which is cofinal.
Together with elements in
$H^0 \cO_P(i_{p, k}, p) = H^0 \cF^{\text{env}}_\cS(H^k_p, L_p)$
and
$H^0 \cO_P(i_{p, k+1}, i_{p, k}) = H^0 \cF^{\text{env}}_\cS(H^{k+1}_p, H^k_p)$
associated with the sequence
$L_p \leadsto H^k_p \leadsto H^{k+1}_p \leadsto \cdots$
in
$\tilde{R}_{L_p}$,
it defines a $P^\prime$-wrapping sequence.
Similarly,
every
$i \in \bZ_{\geq 0}$
belongs to some
$P^\prime$-wrapping sequence.
Finally,
choose additional Floer data for all totally ordered tuples of
$P^\prime$
by induction on the skeleta of the nerve of
$P^\prime$.
\end{proof}

%Rmk
\begin{remark}
As mentioned in
\cite[Rem. 2.19]{GPS2},
condition
(iv)
in Definition
\pref{dfn:2.11}
often holds for countable
$\bfK$.
Then from
\pref{lem:2.18}
one can remove the assumption on
$P$
to be cofinite.
\end{remark}

%Rmk
\begin{remark}
When
$P$
has no duplicates,
one can make
$P^\prime$
do so.
Indeed,
take
$H^k_P$
by induction on
$i \in \bZ_{\geq 0}$
so that
it differs from all the Lagrangians associated with
$P \sqcup \{ 0, \ldots, i_{p, k} - 1 \}$.
Note that
there are uncountably many choices of each
$H^k_P$
by datum
(iv)
in Definition
\pref{dfn:2.11}.
\end{remark}

%Ctn
\begin{caution} \label{ctn:2}
Geometrically,
the nonnegatively wrapping category
$\tilde{R}^{(X, \frakf)}_L$
consists of uncountably many distinct objects.
However,
as written,
Definition
\pref{dfn:2.11}
does not exclude the case
where
$\tilde{R}_L$
consists of a single element
$\id_L$.
In this case,
one cannot find the Lagrangians
$H^k_p$
with the preferred property in the proof of
\pref{lem:2.18}.
If condition
(iv)
in Definition
\pref{dfn:2.11}
allowed the case
where
$\bfK$
consists of a single Lagrangian
$K \in \cL$,
then
$\id_L$
factorizes as
$L \leadsto L^w \leadsto L$
for uncountably many distinct
$L^w \in \cL$
with
$(L^w, K) \in \cL_1$.
But still Definition
\pref{dfn:2.11}
does not tell us
whether we have
$(L^w, L) \in \cL_1$.
Hence we do not know
whether the corresponding elements in
$P$
can have the desired order relation.
\end{caution}

Note that
$P_{\text{univ}}(\cS)$
is the colimit of the filtered system of its sufficiently wrapped subposets with no duplicates
and
downward closed inclusions.
As taking
colimit 
and
directed categories
intertwine,
from
\pref{lem:2.17}
it follows

%Lem
\begin{lemma}[{\cite[Lem. 2.21]{GPS2}}] \label{lem:2.21}
There exists a canonical equivalence
\begin{align} \label{eq:2.21}
H \cW_{P_{\text{univ}}(\cS), \tilde{C}}
\simeq
H \cW_{\DF, \cS, \tilde{C}}.
\end{align}
\end{lemma}

\subsection{Functoriality}
Recall that
a morphism of abstract wrapped Floer setups
\begin{align} \label{eq:morphism}
(\cS, H \cF^{\text{env}}_\cS, \tilde{C}, \{ \tilde{R}_L \}_L)
\to
(\cS^\prime, H \cF^{\text{env}}_{\cS^\prime}, \tilde{C}^\prime, \{ \tilde{R}^\prime_L \}_L)
\end{align}
is given by a commutative diagram
\begin{align*}
\begin{gathered}
\xymatrix{
\cF^{\text{pre}}_\cS \ar@{^{(}->}[d] & H \cF^{\text{pre}}_\cS \ar@{^{(}->}[r] \ar@{^{(}->}[d] & H \cF^{\text{env}}_\cS \ar[r] \ar@{^{(}->}[d] & H \cW_{\DF, \cS, \tilde{C}} \ar[d] \\
\cF^{\text{pre}}_{\cS^\prime} & H \cF^{\text{pre}}_{\cS^\prime} \ar@{^{(}->}[r] & H \cF^{\text{env}}_{\cS^\prime} \ar[r] & H \cW_{\DF, \cS^\prime, \tilde{C}^\prime},
}
\end{gathered}
\end{align*}
where
the horizontal arrows are canonical ones
and
the vertical arrows are given as follows.
\begin{itemize}
\item
The morphism
$\cF^{\text{pre}}_\cS \hookrightarrow \cF^{\text{pre}}_{\cS^\prime}$
is an inclusion of the associated semisimplicial sets covered by identifications of the graded modules for composable pairs compatible with $A_\infty$-operations.
\item 
The morphism
$H \cF^{\text{pre}}_\cS
\hookrightarrow
H \cF^{\text{pre}}_{\cS^\prime}$
is canonically induced by
$\cF^{\text{pre}}_\cS \hookrightarrow \cF^{\text{pre}}_{\cS^\prime}$.
\item 
The morphism
$H \cF^{\text{env}}_\cS
\hookrightarrow
H \cF^{\text{env}}_{\cS^\prime}$
is induced by
$H \cF^{\text{pre}}_\cS
\hookrightarrow
H \cF^{\text{pre}}_{\cS^\prime}$
and
a choice of envelopes
$H \cF^{\text{env}}_\cS, H \cF^{\text{env}}_{\cS^\prime}$
with
$H \cF^{\text{env}}_{\cS^\prime} |_\cS = H \cF^{\text{env}}_\cS$.
\item 
The morphism
$H \cW_{\DF, \cS, \tilde{C}}
\to
H \cW_{\DF, \cS^\prime, \tilde{C}^\prime}$
is uniquely determined via
\pref{lem:2.14}
by the composition
$H \cF^{\text{env}}_\cS
\hookrightarrow
H \cF^{\text{env}}_{\cS^\prime}
\to
H \cW_{\DF, \cS^\prime, \tilde{C}^\prime}$
up to unique equivalence.
\end{itemize}
Note that
the morphism
\pref{eq:morphism}
is defined without using decorated posets.
We denote by
$\operatorname{AWFS}$
the ordinary category of abstract wrapped Floer setups.
Since the canonical naive inclusion
$\cW_{P_{\text{univ}}(\cS), \tilde{C}} \hookrightarrow \cW_{P_{\text{univ}}(\cS^\prime), \tilde{C}^\prime}$
is compatible with composition,
we obtain a functor 
\begin{align*}
\operatorname{AWFS} \to A_\infty Cat^s
\end{align*}
which induces an $\infty$-functor
\begin{align*}
N(\operatorname{AWFS}) \to \cC at^s_{A_\infty}.
\end{align*}

%Rmk
\begin{remark}
In
\cite[Thm. 1.26]{LST},
the above functor was elaborated to an $\infty$-functor
\begin{align*}
\cL iou^\diamond \to \cC at^s_{A_\infty}
\end{align*}
from the $\infty$-category of stabilized Liouville sectors.
\end{remark}

%%%%%%%%%%%%%%%%%%%%%%%%%%%%%%%%%%%%%%%
\section{Modified construction}
In this section,
we modify the abstract wrapped Floer setups.
This amounts to specifying a suitable set of continuation maps.
After introducing the universal envelope,
we define the wrapped Fukaya category as its localization along the continuation maps.

\subsection{Axiomatic composability}
%Def
\begin{definition}[cf.{\cite[Def. 2.9]{GPS2}}] \label{dfn:wcS}
A
\emph{weak abstract Floer setup}
$w \cS$
is an abstract Floer setup
$\cS$
with its datum
(ii)
imposed additional axiomatic composability condition.
Namely,
we replace it with the following datum.
\begin{itemize}
\item[(ii)]
Subsets
$\cL_k \subset \cL^{k+1}$
of
\emph{composable tuples}
for integers
$k \geq 1$
which are closed under
passing to subsequences
and
permuting Lagrangians:
if
$(L_0, \ldots, L_k) \in \cL_k$
then
$(L_{i_0}, \ldots, L_{i_l}) \in \cL_l$
and
$(L_{\sigma(0)}, \ldots, L_{\sigma(k)}) \in \cL_k$
for every
$0 \leq i_0 < \cdots <i_l \leq k$
and
$\sigma \in \frakS_k$.
Moreover,
every element
$(L_0, \ldots, L_k) \in \cL_k$
consists of pairwise distinct Lagrangians.
In particular,
$(L, L) \notin \cL_1$
for all
$L \in \cL$.
\end{itemize}
\end{definition}

%Rmk
\begin{remark} \label{rmk:1}
Different from the original definition,
here in datum
(ii)
we impose the additional conditions on the 
closedness under permutations
and
pairwise distinction.
Geometrically,
a tuple
$(L_0, \ldots, L_k)$
is composable
if it is mutually transverse
\cite[Def. 2.25]{GPS2}.
Hence our additional conditions should not affect the theory.
\end{remark}

%Rmk
\begin{remark} \label{rmk:2}
Perhaps one might wonder
why here we add the closedness under permutations,
as technically it will play almost no role in the sequel.
Moreover,
the original
$\cS$
seems more flexible than our
$w \cS$
in the sense that
one can simultaneously have
$(L, K) \in \cL_1, (K, L) \notin \cL_1$
for some Lagrangians
$L, K \in \cL$,
although it is a geometrically imaginary situation.
However,
that actually means
$\cS$
being more directed rather than flexible:
there are no nontrivial representatives of elements in
$I_{P, \tilde{C}}(q, p)$
for
$p, q \in P$
with
$\xi(p) = L, \xi(q) = K$
and
hence
$I_{P, \tilde{C}}(q, p) = 0$
for all
$P$,
$H^0 \cF^{\text{env}}_\cS$
and
$\tilde{C}$.
In particular,
since all morphisms between non-composable pairs (if exist) will be thrown away,
essentially the original
$\tilde{C}$
has more constraints than our
$C$
coming from the geometrically imaginary hypothesis
$(L, K) \in \cL_1, (K, L) \notin \cL_1$
on
$\cS$.
We package the information on such directions into
$C$
together with
the irrelevant choices
and
the factorization axiom
which will be removed from the original
$(\cS, H \cF^{\text{env}}_\cS, \tilde{C}, \{ \tilde{R}_L \}_L )$.
\end{remark}

We denote by
$\cF^{\text{pre}}_{w \cS}$
the $A_\infty$-pre-category specified by data
(i)-(iv)
and
(vi)
in Definition
\pref{dfn:wcS}.
From
$\cF^{\text{pre}}_{w \cS}$
we obtain the
\emph{weak Donaldson--Fukaya pre-category}
$h \cF^{\text{pre}}_{\DF, w \cS}$
of
$w \cS$
in the same way as Definition
\pref{dfn:2.10}.

\subsection{Canonical envelopes}
Fix a compatible collection
$\delta$
of Floer data.
Specifically,
choose one element
$\delta(L, K) \in D(L, K)$
for each
$(L, K) \in \cL_1$.
By datum
(v)
in
Definition
\pref{dfn:wcS}
one can find an element
$\delta(L, K, M) \in D(L, K, M)$
for each
$(L, K, M) \in \cL_2$
such that
$\mu^2_{\delta(L, K, M)}$
satisfies the $A_\infty$-relation with
$\mu^1_{\delta(L, K)}, \mu^1_{\delta(K, M)}$
and
$\mu^1_{\delta(L, M)}$.
Inductively,
choose one element
$\delta(L_0, \ldots, L_k) \in D(L_0, \ldots, L_k)$
for each
$(L_0, \ldots, L_k) \in \cL_k$
such that
$\mu^k_{\delta(L_0, \ldots, L_k)}$
satisfies the $A_\infty$-relations with respect to the restriction maps on
$D$.
The chosen elements form
$\delta$.

%Def
\begin{definition}
The
\emph{canonical envelope with respect to
$\delta$}
for
$\cF^{\text{pre}}_{w \cS}$
is an $A_\infty$-category
$\cF_{w \cS, \delta}$
whose objects are Lagrangians
$L \in \cL$
and
for every
$L, K \in \cL$
whose morphisms are
\begin{align*}
\cF_{w \cS, \delta}(L, K) =
\begin{cases}
CF(L, K)  & (L, K) \in \cL_1, \\
\bZ \langle 1_L \rangle & L = K, \\
0 & \text{else}.
\end{cases}
\end{align*}
Here,
the $A_\infty$-operations for composable tuples
$(L_0, \ldots, L_k) \in \cL_k$
are
$\mu^k_{\delta(L_0, \ldots, L_k)}$.
Requiring all
$1_L$
to be a strict unit,
they canonically extend to that for any tuples of Lagrangians in
$\cL$.
\end{definition}

By
construction
and
datum
(vi)
in Definition
\pref{dfn:wcS},
the $A_\infty$-category
$\cF_{w \cS, \delta}$
is
cofibrant strictly unital.
Note that
the underlying graded modules of morphism complexes
$\cF_{w \cS, \delta_1}(L, K), \cF_{w \cS, \delta_2}(L, K)$
coincide for any two compatible collections
$\delta_1, \delta_2$
of Floer data.
They just carry different differentials
$\mu^1_{\delta_1(L, K)}, \mu^1_{\delta_2(L, K)}$.
In fact,
we have

%Lem
\begin{lemma} \label{lem:can}
For any compatible collections
$\delta_1, \delta_2$
of Floer data,
the canonical envelopes
$\cF_{w \cS, \delta_1}, \cF_{w \cS, \delta_2}$
with respect to them become equivalent in
$\cC at^s_{A_\infty}$.
\end{lemma}
\begin{proof}
As explained in Definition
\pref{dfn:2.10},
by data
(vii)-(ix)
in  Definition
\pref{dfn:wcS},
for
$(L, K) \in \cL_1$
there are
$\delta^\prime_{12}(L, K), \delta^\prime_{21}(L, K) \in D^\prime(L, K)$
which determine chain maps
$\alpha_{\delta^\prime_{12}(L, K)}, \alpha_{\delta^\prime_{21}(L, K)}$.
Passing to homotopy,
they become mutual inverses up to unique homotopy.
By construction of
$\cF_{w \cS, \delta_1}, \cF_{w \cS, \delta_2}$
this extends to all their morphism complexes.
Hence we obtain a unique equivalence
$h \cF_{\DF, w \cS, \delta_1} \simeq h \cF_{\DF, w \cS, \delta_2}$
of the canonical envelopes with respect to
$\delta_1, \delta_2$
for
$h \cF^{\text{pre}}_{\DF, w \cS}$,
which in turn induces a unique equivalence
$H \cF_{w \cS, \delta_1} \simeq H \cF_{w \cS, \delta_2}$.
Since
$\cF_{w \cS, \delta_1}, \cF_{w \cS, \delta_2}$
are
cofibrant (strictly) unital,
by
\pref{lem:lift}
it lifts to an equivalence
$\cF_{w \cS, \delta_1} \simeq \cF_{w \cS, \delta_2}$
in
$\cC at^s_{A_\infty}$
via
\pref{eq:A}.
\end{proof}

From the above proof one sees that
the canonical envelope
$H \cF_{w \cS}$
for
$H \cF^{\text{pre}}_{w \cS}$
is well defined up to unique equivalence.

%Rmk
\begin{remark}
In view of
Remark
\pref{rmk:1},
\pref{rmk:2}
we may and will assume
$H \cF_{w \cS} \subset H \cF^{\text{env}}_\cS$
throughout the paper.
Everything will work without this assumption
when one replaces
$\tilde{C} |_{H^0 \cF_{w \cS}}$
with
$\tilde{C} |_{H^0 \cF_{w \cS} \cap H^0 \cF^{\text{env}}_\cS}$.
\end{remark}

\subsection{Weak abstract wrapped Floer setups}
Now,
we are ready to remove
the irrelevant choices of
envelope
and
wrapping categories,
as well as
the factorization axiom
from the original abstract wrapped Floer setups.

%Def
\begin{definition}[cf. {\cite[Def. 2.11]{GPS2}}] \label{dfn:wrapped}
A
\emph{weak abstract wrapped Floer setup}
$(w \cS, C)$
is a weak abstract Floer setup
$w \cS$
together with a set
$C \subset H^0 \cF_{w \cS}$
of
\emph{continuation maps}
satisfying the following conditions.
\begin{itemize}
\item[(i)]
All units in
$H \cF_{w \cS}$
are contained in
$C$.
\item[(ii)]
The set
$C$
is closed under composition.
\item[(iii)]
Each diagram
$K \xrightarrow{g} L \xleftarrow{c_{L^w, L}} L^w$
in
$H \cF_{w \cS}$
with
$c_{L^w, L} \in C$
extends to a commutative diagram
\begin{align*}
\begin{gathered}
\xymatrix{
K^w  \ar^{g^w}[r] \ar_{c_{K^w, K}}[d] & L^w \ar^{c_{L^w, L}}[d] \\
K \ar^{g}[r] & L
}
\end{gathered}
\end{align*}
in
$H \cF_{w \cS}$
with
$c_{K^w, K} \in C$.  
\item[(iv)]
If
$g_1, g_2 \colon L \to K$
are morphisms in
$H \cF_{w \cS}$
with
$c_{K, K^\prime} \circ g_1 = c_{K, K^\prime} \circ g_2$
for some 
$c_{K, K^\prime} \colon K \to K^\prime$
in
$C$,
then there exists a map
$c_{L^w, L} \colon L^w \to L$
in
$C$
satisfying
$g_1 \circ c_{L^w, L} = g_2 \circ c_{L^w, L}$.
\item[(v)]
The
\emph{continuation slice category}
$(H^0 \cF_{w \cS})_{/_C L}$,
the full subcategory of the slice category
$(H^0 \cF_{w \cS})_{/ L}$
spanned by continuation maps,
contains at least one object
$L^w \to L$
with
$L^w \neq L$.
\item[(vi)]
The opposite
$((H^0 \cF_{w \cS})_{/_C L})^{op}$
has cofinal countability.
\end{itemize}
\end{definition}

We impose conditions
(i)-(iv)
to make
$C$
a right multiplicative system of
$H \cF_{w \cS}$.
Here,
an analogous result of
\pref{lem:2.14}
immediately follows from
\cite[Prop. 1.3(iv)]{GZ}.
Note that
by Caution
\pref{ctn:2}
a counterpart of condition
(v)
must have been included in Definition
\pref{dfn:2.11}.
The right multiplicative system
$C$
together with condition
(vi)
implies the existence of a
\emph{wrapping category}
$R_L$
for all
$L \in \cL$
in the sense of
\cite[Def. 2.11(iii)]{GPS2},
which exhibits the right locality property.

%Lem
\begin{lemma}[cf. {\cite[Rmk. 2.12]{GPS2}}] \label{lem:wrapping}
Given any set
$C \subset H^0 \cF_{w \cS}$
of continuation maps satisfying conditions
(i)-(iv),
for all
$L \in \cL$
there exists a canonical filtered category
$R_L$
together with a canonical equivalence
$R_L \to ((H^0 \cF_{w \cS})_{/_C L})^{op}$
which exhibits the right locality property.
If in addition
$C$
satisfies condition
(vi),
then
$R_L$
has countable cofinality.
\end{lemma}
\begin{proof}
To each continuation map
$c_{L^w, L} \colon L^w \to L$
we associate a diagram 
$L \leadsto L^w$.
For every such diagrams
$L \leadsto L^w, L \leadsto L^{w^\prime}$
associated with
$c_{L^w, L}, c_{L^{w^\prime}, L} \in C$,
we define a morphism
$(L \leadsto L^w) \to  (L \leadsto L^{w^\prime})$
to be a diagram
$L^w \leadsto L^{w^\prime}$
associated with some
$c_{L^{w^\prime}, L^w} \in C$
satisfying
$c_{L^w, L} \circ c_{L^{w^\prime}, L^w} = c_{L^{w^\prime}, L}$.
Let
$R_L$
be an ordinary category formed by these
diagrams
and
morphisms.
The facts that
every object has unit
and
composition is well defined
respectively follow from conditions
(i)
and
(ii).
Since
$(H^0 \cF_{w \cS})_{/_C L}$
contains
$\id_L$,
the category
$R_L$
is nonempty.
Then by condition
(iii), (iv)
it is filtered.
For any map
$c_{K, K^\prime} \colon K \to K^\prime$
in
$C$,
the induced map
\begin{align*}
HW_{w \cS, C}(L, K)
\coloneqq
\varinjlim_{L^w \in R_L} H \cF_{w \cS}(L^w, K)
\to
HW_{w \cS, C}(L, K^\prime)
\coloneqq
\varinjlim_{L^w \in R_L} H \cF_{w \cS}(L^w, K^\prime)
\end{align*}
is surjective by condition
(iii),
while condition
(iv)
implies its injectivity.
By definition there is a canonical equivalence
$R_L \to ((H^0 \cF_{w \cS})_{/_C L})^{op}$.
Then condition
(vi)
just rephrases that
$R_L$
has countable cofinality.
\end{proof}

%Rmk
\begin{remark} \label{rmk:factorization}
In view of
\pref{lem:wrapping},
the factorization axiom
\cite[Def. 2.11(iv)]{GPS2}
corresponds to the following condition.
\begin{itemize}
\item[(vii)]
Every morphism
$L^{w^\prime} \to L^w$
in
$(H^0 \cF_{w \cS})_{/_C L}$
must satisfy the
\emph{factorization property}:
for all finite sets
$\bfK = \{ (K^j_1, \ldots, K^j_{a_j}) \}_j$
of composable tuples,
there exists a factorization
$L^{w^\prime} \to H \to L^w$
in
$(H^0 \cF_{w \cS})_{/_C L}$
such that
each
$(H, K^j_1, \ldots, K^j_{a_j})$
belongs to
$\cL_{a_j}$.
Moreover,
we require that
there always are uncountably many possibilities for such an
$H$.
\end{itemize}
%As long as there is at least one composable pair
%$(K, M) \in \cL_1$,
%which is in geneal the case,
%via the factorization property one obtains uncountably many objects
%$L^w \to L$
%in
%$(H^0 \cF_{w \cS})_{/_C} L$
%with
%$L^w \neq L$
%from the identity 
%$\id_L$.
\end{remark}

%Rmk
\begin{remark}
Here,
from the original definition
we remove the choices of
envelope
$H \cF^{\text{env}}_\cS$
for
$H \cF^{\text{pre}}_\cS$
and
wrapping categories
$R_L$
for all
$L \in \cL$,
as well as the factorization axiom
\cite[Def. 2.11(iv)]{GPS2}.
Note that
what we add to
$H \cF^{\text{pre}}_{w \cS}$
to obtain
$H \cF_{w \cS}$
is minimal requirement to make
$H \cF^{\text{pre}}_{w \cS}$
an ordinary $\bZ$-linear category,
rather than a choice of data.
As mentioned above,
a counterpart of condition
(v)
must have been included in Definition
\pref{dfn:2.11}.
Thus essentially we do not add anything to the original abstract wrapped Floer setups.
In particular,
for the modified axiomatic construction below,
usual geometric data will always give sufficient input to define the wrapped Fukaya categories.
\end{remark}

For
$(w \cS, C)$
we define the
\emph{weak wrapped Donaldson--Fukaya category}
$H \cW_{\DF, w \cS, C}$
in the same way as Definition
\pref{dfn:2.13}.
By definition of
$H \cF_{w \cS}$
the set
$C$
uniquely determines a collection of sets
$C_\delta \subset H^0 \cF_{w \cS, \delta}$
of continuation maps satisfying conditions
(i) - (vi)
for
$H^0 \cF_{w \cS, \delta}$.
We denote by
$\cW_{w \cS, C_\delta}$
the localization
$\cF_{w \cS, \delta}[C^{-1}_\delta]$.
Each
$\cW_{w \cS, C_\delta}$
gives an $A_\infty$-categorical lift of
$H \cW_{\DF, w \cS, C}$.

%Lem
\begin{lemma} \label{lem:cohomology}
Given any compatible collection
$\delta$
of Floer data,
there is a canonical equivalence
\begin{align*}
H \cW_{w \cS, C_\delta} \simeq H \cW_{\DF, w \cS, C}.
\end{align*}
\end{lemma}
\begin{proof}
Take any two Lagrangians
$L_\delta, K_\delta \in \cL$.
By
\pref{lem:wrapping}
and
condition
(vi)
in Definition
\pref{dfn:wrapped},
there exists a cofinal functor
$\bZ_{\geq 0} \to R_{L_\delta}$.
Let
$L_\delta
=
L^{w_0}_\delta
\leadsto
L^{w_1}_\delta
\leadsto
\cdots$
be the image of the functor,
which corresponds to a sequence
$\cdots
\xrightarrow{c_2}
L^{w_1}_\delta
\xrightarrow{c_1}
L^{w_0}_\delta$
in
$H^0 \cF_{w \cS, \delta}$.
Since
$R_{L_\delta} \simeq ((H^0 \cF_{w \cS, \delta})_{/_{C_\delta} L_\delta})^{op}$
exhibits the right locality property,
by
\pref{thm:5.14}
the induced morphism
\begin{align*}
\varinjlim_i \hom_{\cF_{w \cS, \delta}}(L^{w_i}_\delta, K_\delta)
\to
\varinjlim_i \hom_{\cW_{w \cS, C_\delta}}(L^{w_i}_\delta, K_\delta)
\end{align*}
is a quasi-isomorphism.
Passing to cohomology,
the left hand side becomes
\begin{align*}
H \cW_{\DF, w \cS, C}(L_\delta, K_\delta)
=
\varinjlim_i H \cF_{w \cS, \delta}(L^{w_i}_\delta, K_\delta)
\end{align*}
since a filtered colimit of cohomology is canonically isomorphic to the cohomology of a filtered homotopy colimit.
Note that
the sequence
$L^{w_0}_\delta
\leadsto
L^{w_1}_\delta
\leadsto
\cdots$
is cofinal
and
$\cF_{w \cS, \delta}$
is cofibrant.
Passing to cohomology,
the right hand side becomes
\begin{align*}
H \cW_{w \cS, C_\delta}(L_\delta, K_\delta)
=
\varinjlim_i H \cW_{w \cS, C_\delta}(L^{w_i}_\delta, K_\delta)
\end{align*}
since all
$c_i$
turn to isomorphisms in
$H \cW_{w \cS, C_\delta}$.
Clearly,
the induced isomorphism
\begin{align} \label{eq:isomorphism}
H \cW_{w \cS, C_\delta}(L_\delta, K_\delta)
\xrightarrow{\sim}
H \cW_{\DF, w \cS, C}(L_\delta, K_\delta)
\end{align}
is functorial.
As
$H \cW_{w \cS, \delta}$
and
$H \cW_{\DF, w \cS, C}$
share the same object set
$\cL$,
the assignments
$L_\delta \mapsto L_\delta$
and
\pref{eq:isomorphism}
define the desired equivalence.
\end{proof}

\subsection{Localization along continuation maps}
%Def
\begin{definition}
A
\emph{decorated semisimplicial set}
$\xi \colon E \to \cF^{\text{pre}}_{w \cS}$
for a weak abstract Floer setup
$w \cS$
is a semisimplicial set
$E$
together with a map
$\xi$
from the nerve of
$E$
to the semisimplicial set associated with
$\cF^{\text{pre}}_{w \cS}$.
The map
$\xi$
sends each simplex
$p_k > \cdots > p_0$
in
$E$
to a pair of a composable tuple
$(L_{p_k}, \ldots, L_{p_0}) \in \cL_k$
and
an element
$\xi(p_k, \ldots, p_0) \in D(L_{p_k}, \ldots, L_{p_0})$.
An
\emph{inclusion}
$E \subset E^\prime$
of decorated semisimplicial sets is an inclusion of underlying semisimplicial sets
such that
the decoration
$\xi$
on
$E$
is the restriction of that
$\xi^\prime$
on
$E^\prime$.
\end{definition}

Let
$\xi \colon E \to \cF^{\text{pre}}_{w \cS}$
be any decorated semisimplicial set for
$w \cS$.
We denote by
$\cF_E$
the cofibrant strictly unital $A_\infty$-category 
whose objects are elements
$p \in E$
and
for every
$p, q \in E$
whose morphisms are
\begin{align*}
\cF_E(p, q) =
\begin{cases}
CF(L_p, L_q)  & (L_p, L_q) \in \cL_1, \\
\bZ \langle 1_p \rangle & p = q, \\
0 & \text{else}
\end{cases}
\end{align*}
where
$L_p = \xi(p), L_q = \xi(q)$.
Here,
the $A_\infty$-operations for $k$-simplices
$p_k > \cdots > p_0$
are
$\mu^k_{\xi(p_k, \ldots, p_0)}$
and
we require all
$1_p$
to be a strict unit.
Let
$C_E$
be the set of morphisms in
$H^0 \cF_E$
which is canonically determined by
$C \subset H^0 \cF_{w \cS}$.
We denote by
$\cW_{E, C}$
the localization
$\cF_E[C^{-1}_E]$.

%Def
\begin{definition}
Let
$\delta$
be a compatible collection of Floer data for
$w \cS$.
Then the
\emph{canonical decorated semisimplicial set with respect to $\delta$}
for
$w \cS$
is the one
$\xi_\delta \colon E_\delta \to \cF^{\text{pre}}_{w \cS}$
induced by
$\cF_{w \cS, \delta}$.
\end{definition}

Here,
we use the same symbol
$\cF^{\text{pre}}_{w \cS}$
to denote the associated semisimplicial set,
while we introduce a new symbol
$E_\delta$
to denote the one associated with
$\cF_{w \cS, \delta}$.

%Eg
\begin{example} \label{eg:F}
The $A_\infty$-category
$\cF_{E_\delta}$
can be canonically identified with the canonical envelope
$\cF_{w \cS, \delta}$
with respect to
$\delta$
for
$\cF^{\text{pre}}_{w \cS}$.
\end{example}

%Def
\begin{definition}
An
\emph{entanglement}
$\xi_n \colon E_n \to \cF^{\text{pre}}_{w \cS}$
for
$w \cS$
is a decorated semisimplicial set inductively constructed as follows.
\begin{itemize}
\item[(0)]
Consider the disjoint union
$\bigsqcup_{\delta} E_\delta$.
For each pair
$(p_\delta, q_{\delta^\prime})$
with
$\delta \neq \delta^\prime, p_\delta \in E_\delta, q_{\delta^\prime} \in E_{\delta^\prime}$
and
$(\xi_\delta(p_\delta), \xi_{\delta^\prime}(q_{\delta^\prime})) \in \cL_1$,
add two edges connecting
$p_\delta, q_{\delta^\prime}$
which are of mutually inverse directions. 
Add higher simplices spanned by the
original simplices
and
additional edges.
Using datum
(v)
in Definition
\pref{dfn:wcS},
choose inductively compatible Floer data for the additional simplices.
We denote by
$\xi_0 \colon E_0 \to \cF^{\text{pre}}_{w \cS}$
the resulting decorated semisimplicial set
and
call it a
\emph{$0$-entanglement}.
\item[(1)]
Consider the disjoint union
$E_0 \sqcup E^\prime_0$
of two distinct $0$-entanglements.
For each pair
$(p_{\delta}, q^\prime_{\delta^\prime})$
with
$p_\delta \in E_\delta \subset E_0, q^\prime_{\delta^\prime} \in E_{\delta^\prime} \subset E^\prime_0$
and
$(\xi_\delta(p_\delta), \xi^\prime_{\delta^\prime}(q^\prime_{\delta^\prime})) \in \cL_1$,
add two edges connecting
$p_\delta, q^\prime_{\delta^\prime}$
which are of mutually inverse directions.
Here,
we allow the case
$\delta = \delta^\prime$.
Add higher simplices spanned by the
original simplices
and
additional edges.
Using datum
(v)
in Definition
\pref{dfn:wcS},
choose inductively compatible Floer data for the additional simplices.
We denote by
$\xi_1 \colon E_1 \to \cF^{\text{pre}}_{w \cS}$
the resulting decorated semisimplicial set
and
call it a
\emph{$1$-entanglement}.
\item[(n)]
Consider the disjoint union
$\bigsqcup^n_{i=1} E^i_0$
of 
$n$
distinct $0$-entanglements.
For each pair
$(p^i_{\delta}, q^j_{\delta^\prime})$
with
$i \neq j, p^i_\delta \in E^i_\delta \subset E^i_0, q^j_{\delta^\prime} \in E^j_{\delta^\prime} \subset E^j_0$
and
$(\xi^i_\delta(p^i_\delta), \xi^j_{\delta^\prime}(q^j_{\delta^\prime})) \in \cL_1$,
add two edges connecting
$p^i_\delta, q^j_{\delta^\prime}$
which are of mutually inverse directions.
Here,
we allow the case
$\delta = \delta^\prime$.
Add higher simplices spanned by the
original simplices
and
additional edges.
Using datum
(v)
in Definition
\pref{dfn:wcS},
choose inductively compatible Floer data for the additional simplices.
We denote by
$\xi_n \colon E_n \to \cF^{\text{pre}}_{w \cS}$
the resulting decorated semisimplicial set
and
call it an
\emph{$n$-entanglement}.
\end{itemize}
\end{definition}

%Def
\begin{definition}
A
\emph{morphism}
$E_m \to E_n$
of entanglements for
$w \cS$
is the inclusion
$E_m \subset E_n$
induced by that
$\bigsqcup^m_{i = 1} E^i_0 \subset \bigsqcup^n_{j = 1} E^{\prime j}_0$
of underlying disjoint unions of $0$-entanglements,
which preserves
$\sqcup_\delta E^1_\delta, \ldots, \sqcup_\delta E^m_\delta$
for all compatible collection
$\delta$
of Floer data in each
$E^1_0, \ldots, E^m_0$.
\end{definition}

By definition morphisms of entanglements can be naively composed to yield another.
Hence
the entanglements for
$w \cS$
and
their morphisms
form a category
$\bfE_{w \cS}$.
Note that
when
$m = n$
the only possible morphism
$E_m \to E_n$
is the identity.
Moreover,
given any two entanglements
$E_m, E_n$,
there is at most one morphism
$E_m \to E_n$.

%Rmk
\begin{remark}
The category
$\bfE_{w \cS}$
is not filtered.
For instance,
take two distinct $2$-entanglements
$E_2, E^\prime_2$
which share a common $1$-subentanglement
$E_1 = E^\prime_1$.
Then by definition there is no entanglement
which contains both
$E_2$
and
$E^\prime_2$.
\end{remark}

%Def
\begin{definition}
The
\emph{universal envelope}
$\cF^{\text{univ}}_{w \cS}$
for
$\cF^{\text{pre}}_{w \cS}$
is the colimit for the $\infty$-functor
\begin{align*}
\epsilon_\cF \colon N(\bfE_{w \cS}) \to \cC at^s_{A_\infty}, \
E \mapsto \cF_E.
\end{align*}
\end{definition}

Recall that
the colimit for
$\epsilon_\cF$
is an initial object of the slice $\infty$-category
$(\cC at^s_{A_\infty})_{\epsilon_\cF /}$.
It is well defined up to unique equivalence,
since
$\cC at_{dg} \simeq \cC at^s_{A_\infty}$
admits all colimits.

%Rmk
\begin{remark}
Note that
initial object is unique only up to unique equivalence.
If we choose an initial object
$\cF^{\text{univ}}_{w \cS, \circ}$
of
$(\cC at^s_{A_\infty})_{\epsilon_\cF /}$,
then we obtain the decorated semisimplicial set
\begin{align*}
\xi^\circ_{\text{univ}}(w \cS)
\colon
E^\circ_{\text{univ}}(w \cS)
\to
\cF^{\text{pre}}_{w \cS}
\end{align*}
by forgetting graded modules
$\cF^{\text{univ}}_{w \cS, \circ}(p, q)$.
Then we have
$\cF^{\text{univ}}_{w \cS, \circ} = \cF_{E^\circ_{\text{univ}}(w \cS)}$.
\end{remark}

%Def
\begin{definition}
The
\emph{wrapped Fukaya category}
$\cW_{w \cS, C}$
of a weak abstract wrapped Floer setup
$(w \cS, C)$
is the colimit for the $\infty$-functor
\begin{align*}
\epsilon_\cW \colon N(\bfE_{w \cS}) \to \cC at^s_{A_\infty}, \
E \mapsto \cW_{E, C}.
\end{align*}
\end{definition}

By definition of
$\cF^{\text{univ}}_{w \cS}$
the subsets
$C_E \subset H^0 \cF_E$
for all
$E \in \bfE_{w \cS}$
uniquely determines a subset
$C_{\text{univ}} \subset H^0 \cF^{\text{univ}}_{w \cS}$.
The compositions
$\cF_E \to \cW_{E, C} \to \cW_{w \cS, C}$
induce a morphism
$\cF^{\text{univ}}_{w \cS} \to \cW_{w \cS, C}$
from the colimit for
$\epsilon_\cF$,
which sends all representatives of any element in
$C_{\text{univ}}$
to isomorphisms in
$\cW_{w \cS, C}$.
Via universality of localization,
we obtain a morphism
$\cF^{\text{univ}}_{w \cS}[C^{-1}_{\text{univ}}] \to \cW_{w \cS, C}$
which is unique up to homotopy.
On the other hand,
the morphisms
$\cF_E[C^{-1}_E] \to \cF^{\text{univ}}_{w \cS}[C^{-1}_{\text{univ}}]$
induce a morphism
$\cW_{w \cS, C} \to \cF^{\text{univ}}_{w \cS}[C^{-1}_{\text{univ}}]$
from the colimit for
$\epsilon_\cW$.
Hence the wrapped Fukaya category
$\cW_{w \cS, C}$
is the localization of the universal envelope
$\cF^{\text{univ}}_{w \cS}$
along continuation maps.

%Lem
\begin{lemma} \label{lem:canonical} 
There exists a canonical equivalence in
$\cC at^s_{A_\infty}$
\begin{align} \label{eq:canonical} \cF^{\text{univ}}_{w \cS}[C^{-1}_{\text{univ}}
]
=
\cW_{w \cS, C}.
\end{align}
\end{lemma}

More explicitly,
the wrapped Fukaya category
$\cW_{w \cS, C}$
is the localization of the canonical envelope
$\cF_{w \cS, \delta}$
with respect to
$\delta$
for
$\cF^{\text{pre}}_{w \cS}$
along continuation maps.

%Lem
\begin{lemma} \label{lem:bridge1}
Any compatible collection
$\delta$
of Floer data for
$w \cS$
canonically defines a contractible edge
$\cW_{E_\delta, C} \to \cW_{w \cS, C}$
in
$\cC at^s_{A_\infty}$.
\end{lemma}
\begin{proof}
Let
$E_n \to E_{n+1}$
be any morphism of entanglements.
The functor
$\epsilon_\cF$
sends it to a naive inclusion
$\cF_{E_n} \hookrightarrow \cF_{E_{n+1}}$.
For every
$p_n, q_n \in E_n$
with
$\xi_n(p_n) = L_{p_n}, \xi_n(q_n) = L_{q_n}$,
the induced map
\begin{align*}
\varinjlim_{p^w_n \in 
((H^0 \cF_{E_n})_{/_{C_{E_n}}} p_n)^{op}} H \cF_{E_n}(p^w_n, q_n)
\to
\varinjlim_{p^w_n \in 
((H^0 \cF_{E_{n+1}})_{/_{C_{E_{n+1}}}} p_n)^{op}}H \cF_{E_{n+1}} (p^w_n, q_n)
\end{align*}
is the identity
$H \cW_{E_n, C}(p_n, q_n) = H \cW_{E_{n+1}, C}(p_n, q_n)$.
Indeed,
let
$L_{p_n} = L^{w_0}_{p_n} \leadsto L^{w_1}_{p_n} \leadsto \cdots$
be the image of the cofinal functor
$\bZ_{\geq 0} \to R_{L_{p_n}}$,
which corresponds to a sequence
$\cdots \xrightarrow{c_2} L^{w_1}_{p_n} \xrightarrow{c_1} L^{w_0}_{p_n}$
in
$H^0 \cF_{w \cS}$.
By construction of entanglements,
it comes from a cofinal sequence
$p_n = p^{w_0}_n < p^{w_1}_n < \cdots$
in
$((H^0 \cF_{E_n})_{/_{C_{E_n}}} p_n)^{op}$,
which in turn is also a cofinal sequence in
$((H^0 \cF_{E_{n+1}})_{/_{C_{E_{n+1}}}} p_n)^{op}$.
Then it follows
\begin{align*}
H \cW_{E_n, C}(p_n, q_n) = \varinjlim_{L^w_{p_n} \in R_{L_{p_n}}} H \cF_{w \cS}(L^w_{p_n}, L_{q_n}) = H \cW_{E_{n+1}, C}(p_n, q_n).
\end{align*}
Hence
$\epsilon_\cW$
sends
$E_n \to E_{n+1}$
to a naive inclusion
$\cW_{E_n, C} \hookrightarrow \cW_{E_{n+1}, C}$.
Let
$p_{n+1} \in E_{n+1}$
be any vertex
whose image
$\xi_{n+1}(p_{n+1}) = L_{p_{n+1}}$
is not a zero object.
By condition
(v)
in Definition
\pref{dfn:wrapped},
there exists at least one map
$c_{L^w_{p_{n+1}}, L_{p_{n+1}}} \colon L^w_{p_{n+1}} \to L_{p_{n+1}}$
in
$C$
with
$(L^w_{p_{n+1}}, L_{p_{n+1}}) \in \cL_1$.
Then by construction of
$E_n$
one always finds some
$p^w_n \in E_n$
with
$\xi_n(p^w_n) = L^w_{p_{n+1}}$.
In particular,
via the map
$c_{p^w_n, p_{n+1}} \colon p^w_n \to p_{n+1}$
in
$C_{E_{n+1}}$
corresponding to
$c_{L^w_{p_n}, L_{p_{n+1}}}$,
two vertices
$p^w_n, p_{n+1}$
become isomorphic in
$\cW_{E_{n+1}, C}$.
Hence the naive inclusion
$\cW_{E_n, C} \hookrightarrow \cW_{E_{n+1}, C}$
is a quasi-equivalence.
By the same argument as above,
one can show that
there is a canonical quasi-equivalence
$\cW_{E_\delta, C} \to \cW_{E_n, C}$
whenever we make a choice of an inclusion into some $n$-entanglement
$E_n$.
Between two such inclusions
$E_\delta \subset E_n, E_\delta \subset E^\prime_n$,
we have the naive natural transformations
which are mutually inverse.
This is compatible with any morphism of entanglements.
It follows that
the canonical functor
$\cW_{E_n, C} \to \cW_{w \cS, C}$
is a quasi-equivalence for all $n$-entanglements.
Hence we obtain a homotopically unique contractible edge
$\cW_{E_\delta, C} \to \cW_{w \cS, C}$
in
$\cC at^s_{A_\infty}$.
\end{proof}

\subsection{Functoriality}
%Def
\begin{definition}[cf. {\cite[Eq. (2.19)]{GPS2}}] \label{dfn:morphism}
A
\emph{morphism of weak abstract wrapped Floer setups}
$(w \cS, C) \to (w \cS^\prime, C^\prime)$
is a commutative diagram
\begin{align*}
\begin{gathered}
\xymatrix{
\cF^{\text{pre}}_{w \cS} \ar@{^{(}->}[d] & C \ar@{^{(}->}[r] \ar@{^{(}->}[d] & H^0 \cF_{w \cS} \ar@{^{(}->}[d] \\
\cF^{\text{pre}}_{w \cS^\prime} & C^\prime \ar@{^{(}->}[r] & H^0 \cF_{w \cS^\prime},
}
\end{gathered}
\end{align*}
where
the horizontal arrows are canonical ones
and
the vertical arrows are given as follows.
\begin{itemize}
\item
The morphism
$\cF^{\text{pre}}_{w \cS} \hookrightarrow \cF^{\text{pre}}_{w \cS^\prime}$
is an inclusion of the associated semisimplicial sets covered by identifications of the graded modules for composable pairs compatible with $A_\infty$-operations.
\item 
The morphism
$H \cF_{w \cS}
\hookrightarrow
H \cF_{w \cS^\prime}$
is canonically induced by
$\cF^{\text{pre}}_{w \cS} \hookrightarrow \cF^{\text{pre}}_{w \cS^\prime}$.
\item 
The morphism
$C \hookrightarrow C^\prime$
is an inclusion with
$C^{\prime} |_{H^0 \cF_{w \cS}} = C$.
\end{itemize}
\end{definition}

Any morphism
$(w \cS, C) \to (w \cS^\prime, C^\prime)$
of weak abstract wrapped Floer setups induces a canonical naive inclusion
$\cW_{E_\delta, C} \hookrightarrow \cW_{E^\prime_{\delta^\prime}, C^\prime}$
for each canonical
decorated semisimplicial set 
$E^\prime_{\delta^\prime}$
with respect to
$\delta^\prime$
for
$w \cS^\prime$,
where
$\delta^\prime$
is a compatible collection of Floer data extending
$\delta$.
Between two such inclusions
$\cW_{E_\delta, C} \hookrightarrow \cW_{E^\prime_{\delta^\prime}, C^\prime}, \cW_{E_\delta, C} \hookrightarrow \cW_{E^{\prime}_{\delta^{\prime \prime}}, C^\prime}$,
we have the naive natural transformations
which are mutually inverse.
This is compatible with the contractible edges
$\cW_{E^\prime_{\delta^\prime}, C^\prime} \to \cW_{w \cS^\prime, C^\prime},  \cW_{E^{\prime}_{\delta^{\prime \prime}}, C^\prime} \to \cW_{w \cS^\prime, C^\prime}$
from
\pref{lem:bridge1}.
Hence we obtain a homotopically unique edge
$\cW_{E_\delta, C} \to \cW_{w \cS^\prime, C^\prime}$
in
$\cC at^s_{A_\infty}$,
which in turn induces via
\pref{lem:bridge1}
an edge
$\cW_{w \cS, C} \to \cW_{w \cS^\prime, C^\prime}$.
Since this is compatible with composition,
we obtain an $\infty$-functor
\begin{align} \label{eq:functorial}
N(\operatorname{wAWFS}) \to \cC at^s_{A_\infty}, \ (w \cS, C) \mapsto \cW_{w \cS, C}
\end{align}
for the ordinary category
$\operatorname{wAWFS}$
of weak abstract wrapped Floer setups.

\subsection{Proof of \pref{thm:main}}
As a summary of this section,
we write down a proof of
\pref{thm:main}.
The colimit
$\cF^{\text{univ}}_{w \cS}$
for the $\infty$-functor
$\epsilon_\cF \colon N(\bfE_{w \cS}) \to \cC at^s_{A_\infty}$
is a cofibrant strictly unital $A_\infty$-category.
Since it is determined by only
$w \cS$,
we obtain the assignments
\pref{eq:assignF}.

(i)
A weak abstract wrapped Floer setup
$(w \cS, C)$
determines
the category
$\bfE_{w \cS}$
of entanglements for
$w \cS$
and
the compatible subsets
$C_E \subset H^0 \cF_E$
for all
$E \in \bfE_{w \cS}$.
By definition of
$\cF^{\text{univ}}_{w \cS}$
the latter uniquely determine a subset
$C_{\text{univ}} \subset H^0 \cF^{\text{univ}}_{w \cS}$.
Hence the assignments
\pref{eq:assignF}
extend to ones
\pref{eq:assignW}.

(ii)
Via
\pref{lem:canonical}
the assignments
\pref{eq:assignW}
define an $\infty$-functor
\pref{eq:functorial}.

(iii)
Fix a compatible collection
$\delta$
of Floer data for
$w \cS$.
By
\pref{lem:cohomology}
there is a canonical equivalence
$H \cW_{E_\delta, C} \xrightarrow{\sim} H \cW_{\DF, w \cS, C}$.
From the proof of
\pref{lem:bridge1}
we obtain a canonical equivalence
$H \cW_{E_\delta, C} \xrightarrow{\sim} H \cW_{w \cS, C}$.
Let
$\delta^\prime$
be any compatible collection of Floer data for
$w \cS^\prime$
extending
$\delta$.
Via the canonical equivalences,
the inclusion
$H \cW_{E_\delta, C} \hookrightarrow H \cW_{E^\prime_{\delta^\prime}, C^\prime}$
induces ones
$H \cW_{w \cS, C} \hookrightarrow H \cW_{w \cS^\prime, C^\prime}, \
H \cW_{\DF, w \cS, C} \hookrightarrow H \cW_{\DF, w \cS^\prime, C^\prime}$.
From construction of the $\infty$-functor
\pref{eq:functor},
it follows that
$H \cW_{E_\delta, C} \hookrightarrow H \cW_{E^\prime_{\delta^\prime}, C^\prime}$
and
$H \cW_{w \cS, C} \hookrightarrow H \cW_{w \cS^\prime, C^\prime}$
respectively coincide with the cohomology of
the inclusion
$\cW_{E_\delta, C} \hookrightarrow \cW_{E^\prime_{\delta^\prime}, C^\prime}$
and
the edge
$\cW_{w \cS, C} \to \cW_{w \cS^\prime, C^\prime}$.

%%%%%%%%%%%%%%%%%%%%%%%%%%%%%%%%%%%%%%%
\section{Proof of \pref{thm:comparison}}
Let
$(w \cS, C)$
be
a weak abstract wrapped Floer setup
and 
$(\cS, H \cF^{\text{env}}_\cS, \tilde{C}, \{ \tilde{R}_L \}_L)$
an abstract wrapped Floer setup with
$\tilde{C} |_{H^0 \cF_{w \cS}} = C$.
For any compatible collection
$\delta$
of Floer data for
$w \cS$,
consider all
$P_\delta \in \Pos_\cS$
which are sufficiently wrapped with respect to
$(\cS, H \cF^{\text{env}}_\cS, \tilde{C}, \{ \tilde{R}_L \}_L)$
and
whose associated Floer data can be obtained from
$\delta$.
Each 
$P_\delta$
defines a canonical strictly unital $A_\infty$-functor
\begin{align*}
\iota_{P_\delta} \colon \cO_{P_\delta} \to \cF_{E_{\delta}}
\end{align*}
sending a vertex to the one corresponding to the same Lagrangian in
$\cL$.
Note that
in general there is no $A_\infty$-functor of opposite direction.
Since
$\iota_{P_\delta}$
sends
$I_{P_\delta, \tilde{C}}$
to
$C_{E_\delta}$,
it induces a canonical strictly unital $A_\infty$-functor
\begin{align*}
\tau_{P_\delta} \colon \cW_{P_\delta, \tilde{C}} \to \cW_{E_\delta, C}.
\end{align*}
Let
$P_{\text{univ}, \delta}(\cS)$
be the filtered union of all
$P_\delta$
with respect to downward closed inclusions.
Then
$\iota_{P_\delta}$
for all
$P_\delta$
induce
\begin{align} \label{eq:iota}
\iota_\delta \colon \cO_{P_{\text{univ}, \delta}(\cS)} \to \cF_{E_\delta},
\end{align}
which in turn induces
\begin{align} \label{eq:tau}
\tau_\delta \colon \cW_{P_{\text{univ}, \delta}(\cS), \tilde{C}} \to \cW_{E_\delta, C}.
\end{align}

%Lem
\begin{lemma} \label{lem:bridge2}
The strictly unital $A_\infty$-functor
\pref{eq:tau}
is a quasi-equivalence.
\end{lemma}
\begin{proof}
First,
we show that
the induced maps
\begin{align} \label{eq:want}
H \cW_{P_\delta, \tilde{C}}(p, q) \to H \cW_{E_\delta, C}(\tau_{P_\delta}(p), \tau_{P_\delta}(q))
\end{align}
by
$\tau_{P_\delta}$
are isomorphisms for all
$p, q \in \cW_{P_\delta, \tilde{C}}$.
Let
$p = p_0 < p_1 < \cdots$
be a $P_\delta$-wrapping sequence of
$p$.
From the proof of
\pref{lem:2.18}
one sees that
it corresponds to some cofinal sequence
$L_p = L_{p_0} \leadsto L_{p_1} \leadsto$
in
$\tilde{R}_{L_p}$.
Note that
$L_p = L_{p_0} \leadsto L_{p_1} \leadsto \cdots$
in turn corresponds to a sequence
$\cdots L_{p_1} \to L_{p_0} = L_p$
of maps in
$I_{P_\delta, \tilde{C}}$.
By
$\tilde{C} |_{H^0 \cF_{w \cS}} = C$,
$\iota_{P_\delta}(I_{P_\delta, \tilde{C}}) \subset C_{E_\delta}$
and
\pref{lem:wrapping}
the cofinal sequence defines one in
$R_{L_p}$.
Since we have
\begin{align*}
\cF_{E_\delta}(\iota_{P_\delta}(p), \iota_{P_\delta}(q)) = CF(L_p, L_q) = \cO_{P_\delta}(p, q)
\end{align*}
as cochain complexes for
$p > q$,
the maps
\pref{eq:want}
factor through the composition of isomorphisms
\begin{align*}
\varinjlim_i H \cW_{P_\delta, \tilde{C}}(p_i, q)
\xrightarrow{\sim}
\varinjlim_i H \cO_{P_\delta}(p_i, q)
=
\varinjlim_i H \cF_{E_\delta}(\iota_{P_\delta}(p_i), \iota_{P_\delta}(q)).
\end{align*}
Here,
the first isomorphism follows from
\pref{lem:2.16}.
Combining with
\pref{thm:5.14},
one sees that
$\tau_{P_\delta}$
is fully faithful.
If there is a downward closed inclusion
$P_\delta \subset P^\prime_\delta$,
then the composition of
$\tau_{P^\prime_{\delta^\prime}}$
with the naive inclusion
$\cW_{P_\delta, \tilde{C}} \hookrightarrow \cW_{P^\prime_{\delta^\prime}, \tilde{C}}$
coincides with
$\tau_{P_\delta}$.
From
\pref{lem:2.17}
it follows that
$\tau_{\delta}$
is fully faithful,
as taking
colimit
and
directed categories
intertwine.
Next,
we show that
$\tau_\delta$
is essentially surjective.
For all object
$p \in \cW_{E_\delta, C}$
consider any wrapping sequence of
$p$
obtained by using Floer data from
$\delta$.
Such a wrapping sequence defines an object
$P_{p, \delta} \in \Pos_\cS$
which are sufficiently wrapped.
Then we have
$\tau_{P_{p, \delta}}(p) = p$.
\end{proof}

Choose an inclusion
$E_\delta \subset E$
to some entanglement.
Then
\pref{eq:iota}
extends to a strictly unital $A_\infty$-functor
\begin{align} \label{eq:iotabar}
\bar{\iota}_\delta \colon \cO_{P_{\text{univ}, \delta}(\cS)} \to \cF_{E_\delta} \to \cF^{\text{univ}}_{w \cS}.
\end{align}
Under the choice
$E_\delta \subset E$,
via the canonical quasi-equivalence
$\cF_{E_\delta}[C^{-1}_{E_\delta}] \to \cF_E[C^{-1}_E]$
it induces a quasi-equivalence
\begin{align} \label{eq:taubar}
\bar{\tau}_\delta \colon \cW_{P_{\text{univ}, \delta}(\cS), \tilde{C}} \to \cW_{E_\delta, C} \to \cF^{\text{univ}}_{w \cS}[C^{-1}_{\text{univ}}]
\end{align}
extending
\pref{eq:tau}.
In other words,
the localizations
$\cO_{P_{\text{univ}, \delta}(\cS)} \to \cW_{P_{\text{univ}, \delta}(\cS), \tilde{C}}, \cF^{\text{univ}}_{w \cS} \to \cF^{\text{univ}}_{w \cS}[C^{-1}_{\text{univ}}]$
fit into a commutative diagram in
$A_\infty Cat^s$
\begin{align*}
\begin{gathered}
\xymatrix{
\cO_{P_{\text{univ}, \delta}(\cS)} \ar^{\bar{\iota}_\delta}[r] \ar[d] & \cF^{\text{univ}}_{w \cS} \ar[d] \\
\cW_{P_{\text{univ}, \delta}(\cS), \tilde{C}} \ar^{\bar{\tau}_\delta}[r] & \cF^{\text{univ}}_{w \cS}[C^{-1}_{\text{univ}}].
}
\end{gathered}
\end{align*}

On the other hand,
by
\pref{thm:5.9}
the localizations satisfy the universal property in the sense of
\cite[Prop. 5.1(a)]{OT25}.
Since all $A_\infty$-categories are cofibrant (strictly) unital,
this means that
for any unital $A_\infty$-category
$\cD$
the pullbacks along the localizations
\begin{align*}
\Fun^u_{A_\infty}(\cW_{P_{\text{univ}, \delta}(\cS), \tilde{C}}, \cD) \to \Fun^u_{A_\infty}(\cO_{P_{\text{univ}, \delta}(\cS)}, \cD), \
\Fun^u_{A_\infty}(\cF^{\text{univ}}_{w \cS}[C^{-1}_{\text{univ}}], \cD) \to \Fun^u_{A_\infty}(\cF^{\text{univ}}_{w \cS}, \cD)
\end{align*}
are fully faithful
and
their essential images consist of those functors
which send all representatives of any elements in 
$I_{P_{\text{univ}, \delta}(\cS), \tilde{C}}, C_{\text{univ}}$
to
isomorphisms in
$\cD$.
Hence we obtain a commutative diagram
\begin{align*}
\begin{gathered}
\xymatrix{
\Fun^u_{A_\infty}(\cF^{\text{univ}}_{w \cS}, \cD) \ar^{\bar{\iota}^*_\delta}[r] & \Fun^u_{A_\infty}(\cO_{P_{\text{univ}, \delta}(\cS)}, \cD) \\
\Fun^u_{A_\infty}(\cF^{\text{univ}}_{w \cS}[C^{-1}_{\text{univ}}], \cD) \ar^{\bar{\tau}^*_\delta}[r] \ar@{^{(}-_{>}}[u] & \Fun^u_{A_\infty}(\cW_{P_{\text{univ}, \delta}(\cS), \tilde{C}}, \cD) \ar@{^{(}-_{>}}[u]
}
\end{gathered}
\end{align*}
in
$A_\infty Cat^u$,
where the bottom horizontal arrow is a quasi-equivalence by
\pref{lem:3.46}.

By
\pref{lem:lift}
the inverse
$H \cF^{\text{univ}}_{w \cS}[C^{-1}_{\text{univ}}] \to H \cW_{P_{\text{univ}, \delta}(\cS), \tilde{C}}$
of
$H \bar{\tau}_\delta$
admits a lift to a unital $A_\infty$-functor
$\cF^{\text{univ}}_{w \cS}[C^{-1}_{\text{univ}}] \to \cW_{P_{\text{univ}, \delta}(\cS), \tilde{C}}$,
which is unique up to homotopy.
We denote by
$\iota \colon \cF^{\text{univ}}_{w \cS} \to \cW_{P_{\text{univ}, \delta}(\cS), \tilde{C}}$
the composition with the localization.
The pullback 
\begin{align*}
\iota^* \colon \Fun^u_{A_\infty}(\cW_{P_{\text{univ}, \delta}(\cS), \tilde{C}}, \cD) \hookrightarrow \Fun^u_{A_\infty}(\cF^{\text{univ}}_{w \cS}, \cD) 
\end{align*}
is fully faithful
and
its essential image consists of those functors
which send all representatives of any elements in $C_{\text{univ}}$
to isomorphisms in
$\cD$.
By
\pref{lem:3.46}
the canonical quasi-equivalence
$\cW_{P_{\text{univ}, \delta}(\cS), \tilde{C}} \xrightarrow{\text{GPS}} \cW_{P_{\text{univ}(\cS)}, \tilde{C}}$
induces a quasi-equivalence
\begin{align*}
\Fun^u_{A_\infty}(\cW_{P_{\text{univ}(\cS)}, \tilde{C}} ,\cD) \xrightarrow{\text{GPS}^*} \Fun^u_{A_\infty}(\cW_{P_{\text{univ}, \delta}(\cS), \tilde{C}} ,\cD).
\end{align*}
Now,
the pullback
\begin{align*} 
\bar{\iota}^* \colon \Fun^u_{A_\infty}(\cW_{P_{\text{univ}}(\cS), \tilde{C}}, \cD) \hookrightarrow \Fun^u_{A_\infty}(\cF^{\text{univ}}_{w \cS}, \cD) 
\end{align*}
for
$\bar{\iota} = \text{GPS} \circ \iota$
gives a unital $A_\infty$-functor with universal property.
Choose a quasi-equivalence
$\cF^{\text{univ}}_{w \cS}[C^{-1}_\text{univ}] \to \cW_{w \cS, C}$
defining the homotopy class of
\pref{eq:canonical}.
Then one can run the same argument as above,
replacing
$\cF^{\text{univ}}_{w \cS}[C^{-1}_{\text{univ}}]$
with
$\cW_{w \cS, C}$.
In particular,
we obtain the desired quasi-equivalence
$\bar{\tau} \colon \cW_{P_{\text{univ}}(\cS), \tilde{C}} \to \cW_{w \cS, C}$
by composing a chosen inverse
$\text{GPS}^{-1}$
of
$\text{GPS}$.

%Rmk
\begin{remark}
Passing to
$\cC at^u_{A_\infty}$,
the above choices of
$E_\delta \subset E$,
$\cF^{\text{univ}}_{w \cS}[C^{-1}_\text{univ}] \to \cW_{w \cS, C}$,
$\tau^{-1}_\delta$
and
$\text{GPS}^{-1}$
become canonical.
This is because different choices of each define the same edge up to homotopy by
\pref{lem:canonical},
\pref{lem:bridge1}
and
the basic property of
$\cC at^s_{A_\infty}$.
\end{remark}

%%%%%%%%%%%%%%%%%%%%%%%%%%%%%%%%%%%%%%%


\begin{thebibliography}{999999}
\bibitem[COS]{COS}
A. Canonaco, M. Ornaghi and P. Stellari,
\emph{Localizations of the category of $A_\infty$-categories and intertnal Homs over a ring},
arXiv:2404.06610

\bibitem[GJ25]{GJ}
B. Gammage and M. Jeffs,
\emph{Homological mirror symmetry for functors between Fukaya categories of very affine hypersurfaces},
J. Topol. 18(1), e70012 (2025)

\bibitem[GL22]{GL}
B. Gamage and I. Le,
\emph{Mirror Symmetry for truncated cluster varieties},
SIGMA 18, 055 (2022).
 
\bibitem[GPS1]{GPS1}
S. Ganatra, J. Pardon and V. Shende,
\emph{Covariantly functorial wrapped Floer theory on Liouville sectors},
Publ. math. IHES. 131, 73–200 (2020).

\bibitem[GPS2]{GPS2}
S. Ganatra, J. Pardon and V. Shende,
\emph{Sectorial descent for wrapped Fukaya categories},
J. Amer. Math. Soc. 37, 499-635 (2024)

\bibitem[GPS3]{GPS3}
S. Ganatra,  J. Pardon and V. Shende,
\emph{Microlocal Morse theory of wrapped Fukaya categories},
Ann. Math. 199(3), 943-1042 (2024)

\bibitem[GS20]{GS1}
B. Gammage and V. Shende,
\emph{Mirror symmetry for very affine hypersurfaces},
Acta. Math. 229(2), 287-346 (2022).

\bibitem[GS23]{GS2}
B. Gammage and V. Shende,
\emph{Homological mirror symmetry at large volume},
Tunis. J. Math. 5(1), 31-71 (2023).

\bibitem[GZ67]{GZ}
P. Gabriel and M. Zisman,
\emph{Calculus of fractions and homotopy theory},
Ergebnisse der Mathematik und ihrer Grenzgeniete,
Band 35,
Springer--Verlag New York,
New York, 1967.

\bibitem[Hov99]{Hov}
M. Hovey,
\emph{Model categories},
Mathematical Surveys and Monographs,
American Mathematical Society (AMS), Providence, RI, 1999.

\bibitem[Kuw20]{Kuw}
T. Kuwagaki,
\emph{The nonequivariant coherent constructible correspondence for toric stacks},
Duke Math. J. 169(11), 2125-2197 (2020).

%\bibitem[Lee]{Lee}
%H. Lee,
%\emph{Homological mirror symmetry for open Riemann surfaces from pair-of-pants decompositions},
%arXiv:1608.04473
 
\bibitem[Lyu03]{Lyu03}
V.V. Lyubashenko,
\emph{Category of $A_\infty$-categories},
Homol. Homotopy Appl. 5(1), 1-48 (2003).

\bibitem[LM08]{LM}
V.V. Lyubashenko and O. Manzyuk,
\emph{Quotients of unital $A_\infty$-categories},
Theory Appl. Categ.  20(13), 405-496 (2008).

\bibitem[LO06]{LO}
V.V. Lyubashenko and S. Ovsienko,
\emph{A construction of quotient $A_\infty$-categories},
Homol. Homotopy Appl. 8(2), 157-203 (2006).

\bibitem[LST]{LST}
O. Lazarev, Z. Sylvan and H. L. Tanaka,
\emph{The infinity-category of stabilized Liouville sectors},
arXiv:2110.11754

\bibitem[Mor]{Mor}
H. Morimura,
\emph{Sectorial covers over fanifolds},
arXiv:2310.10084v2

\bibitem[MSZ]{MSZ}
H. Morimura, N. Sibilla and P. Zhou,
\emph{Homological mirror symmetry for complete intersections in algebraic tori},
arXiv:2303.06955v3

\bibitem[OT25]{OT25}
Y. G. Oh and H. L. Tanaka,
\emph{$\infty$-categorical universal properties of quotients and localizations of $A_\infty$-categories},
to appear in Homol. Homotopy Appl.

\bibitem[Pas24]{Pas}
J. Pascaleff,
\emph{Remarks on the equivalence between differential graded categories and A-infinity categories},
Homol. Homotopy Appl. 26(1), 275-285 (2024).

\bibitem[PS1]{PS1}
J. Pascaleff and N. Sibilla,
\emph{Topological Fukaya category and mirror symmetry for punctured surfaces},
Compos. Math. 155(3), 599-644 (2019).

\bibitem[PS2]{PS2}
J. Pascaleff and N. Sibilla,
\emph{Fukaya categories of higher-genus surfaces and pants decompositions},
arXiv:2103.03366v3

\bibitem[PS3]{PS3}
J. Pascaleff and N. Sibilla,
\emph{Singularity categories of normal crossings surfaces, descent, and mirror symmetry}, arXiv:2208.03896
 
\bibitem[Sei08]{Sei}
P. Seidel,
\emph{Fukaya categories and Picard--Lefschetz theory},
Zurich Lectures in Advanced Mathematics,
European Mathematical Society (EMS),
Z\"{u}rich, 2008.

\bibitem[Spa88]{Spa}
N. Spaltenstein,
\emph{Resolutions of unbounded complexes},
Compos. Math. 65(2), 121-154 (1988).

\bibitem[Tan]{Tan}
H. L. Tanaka,
\emph{Unitalities and mapping spaces in $A_\infty$-categories},
arXiv:2407.05532
\end{thebibliography}
\end{document}